\documentclass[british,11pt]{amsart}
\pdfoutput=1

\usepackage{amsmath, amsfonts, amsthm, amssymb, multicol}
\usepackage{graphicx}
\usepackage{float}
\usepackage{verbatim}%

\makeatletter
\@namedef{subjclassname@2020}{%
  \textup{2020} Mathematics Subject Classification}
\makeatother

\numberwithin{equation}{section}

\allowdisplaybreaks

\usepackage{geometry}

\calclayout

\usepackage[T1]{fontenc}
\usepackage[utf8]{inputenc}
\usepackage{babel}
\usepackage{csquotes}
\usepackage{mathtools}
\usepackage{microtype}

\newtheorem{thm}{Theorem}[section]
\newtheorem*{theorem*}{Theorem}
\newtheorem{lma}[thm]{Lemma}
\newtheorem{cor}[thm]{Corollary}
\newtheorem{defn}[thm]{Definition}

\newtheorem{prop}[thm]{Proposition}

\newtheorem{rem}[thm]{Remark}
\newtheorem{ques}[thm]{Question}

\newtheorem{example}[thm]{Example}

\usepackage[backend=biber,backref=true,doi=false,
url=false,isbn=false,date=year,mincrossrefs=9,
giveninits=true,maxbibnames=99,sortcites=true,
hyperref=true]{biblatex}
\renewbibmacro{in:}{%
  \ifentrytype{article}{}{\printtext{\bibstring{in}\intitlepunct}}}
\DeclareFieldFormat[article,periodical]{volume}{\mkbibbold{#1}}%
\DeclareFieldFormat[article,periodical,unpublished]{citetitle}{#1}
\DeclareFieldFormat[article,periodical,unpublished]{title}{#1}
\DeclareFieldFormat{pages}{#1}
\DeclareFieldFormat[inproceedings]{title}{#1\isdot}
\DeclareFieldFormat[misc]{title}{#1\isdot}
\DeclareFieldFormat[misc,article]{note}{#1\addcomma}

\AtEveryBibitem{
  \clearfield{number}}
\newcommand{\N}{\mathbb{N}}
\newcommand{\Rd}{\mathbb{R}^d}
\newcommand{\ubd}{\overline{\dim}_{\mathrm{B}}}
\newcommand{\bd}{\dim_{\mathrm{B}}}
\newcommand{\lbd}{\underline{\dim}_{\mathrm{B}}}
\newcommand{\uid}{\overline{\dim}_{\,\theta}}
\newcommand{\lid}{\underline{\dim}_{\,\theta}}
\newcommand{\upd}{\overline{\dim}^{\Phi}}
\newcommand{\lpd}{\underline{\dim}^{\Phi}}

\newcommand{\fix}{\mathrm{fix}}

\renewcommand{\epsilon}{\varepsilon}

\renewcommand{\geq}{\geqslant}
\renewcommand{\leq}{\leqslant}

\title{Intermediate dimensions of infinitely generated attractors}

\author{Amlan Banaji}
\address{Amlan Banaji, University of St Andrews, UK}
\email{afb8@st-andrews.ac.uk}

\author{Jonathan M. Fraser}
\address{Jonathan M. Fraser, University of St Andrews, UK}
\email{jmf32@st-andrews.ac.uk}

\begin{document}

\date{}

\newgeometry{top=0.7in,bottom=0.4in,left=1.2in,right=1.2in}

\begin{abstract}
We study the dimension theory of limit sets of iterated function systems consisting of a countably infinite number of contractions. Our primary focus is on the \emph{intermediate dimensions}: a family of dimensions depending on a parameter $\theta \in [0,1]$ which interpolate between the Hausdorff and box dimensions. Our main results are in the case when all the contractions are conformal. Under a natural separation condition we prove that the intermediate dimensions of the limit set are the maximum of the Hausdorff dimension of the limit set and the intermediate dimensions of the set of fixed points of the contractions. This builds on work of Mauldin and Urba\'nski concerning the Hausdorff and upper box dimension.  We give several (often counter-intuitive) applications of our work to dimensions of projections, fractional Brownian images, and general H\"older images. These applications apply to well-studied examples such as sets of numbers which have real or complex continued fraction expansions with restricted entries.

We also obtain several results without assuming conformality or any separation conditions. We prove general upper bounds for the Hausdorff, box and intermediate dimensions of infinitely generated attractors in terms of a topological pressure function. We also show that the limit set of a `generic' infinite iterated function system has box and intermediate dimensions equal to the ambient spatial dimension, where `generic' can mean either `full measure' or `comeagre.' \vspace{-.3cm}
\end{abstract}

\keywords{infinite iterated function system, conformal iterated function system, intermediate dimensions, Hausdorff dimension, box dimension, topological pressure function, continued fractions \\
\indent \emph{Journal ref.:} Trans. Amer. Math. Soc. \textbf{376} (2023), 2449–2479}

\subjclass[2020]{28A80 (Primary), 37B10, 11K50 (Secondary)}

\maketitle

\tableofcontents
\restoregeometry 

\section{Introduction}

\subsection{Background}

An iterated function system (IFS) is a finite set of contractions $\{ S_i \colon X \to X\}_{i \in I}$, where $X$ is a closed subset of Euclidean space. To each IFS one can define an associated limit set (or attractor) $F$ satisfying $F = \cup_{i \in I} S_i(F)$, which will often be fractal in nature. IFSs and the dimension theory of the associated limit sets have been studied extensively since Hutchinson's important paper~\cite{Hutchinson1981}. In a seminal 1996 paper~\cite{Mauldin1996} Mauldin and Urbański extended the theory to infinite iterated function systems (IIFSs) consisting of countably many contractions, with the contraction ratios uniformly bounded above by some $\rho < 1$.  The dimension theory of IIFSs has been studied further in~\cite{Mauldin1999,Ngai2016,Chu2020,Kaenmaki2014,Mauldin1995:Fractals1995} and many other works. Mauldin and Urbański showed that there are many similarities, but also many differences, between finite and infinite iterated function systems, and paid particular attention to (infinite) conformal iterated function systems (CIFSs, defined in Definition~\ref{cifs}), where the contractions are conformal and are sufficiently separated. One notable difference is that the two most familiar notions of fractal dimension, namely the Hausdorff and box dimensions (defined in Section~\ref{definedimsec}), coincide for the limit set of any finite CIFS, but can differ for infinite CIFSs, which indicates that the presence of infinitely many maps can cause the limit set to have greater inhomogeneity in space. 
In particular, Mauldin and Urbański showed that for a CIFS the Hausdorff dimension can be determined from a certain topological pressure function defined in~\eqref{MUpressure} below (see \cite[Theorem~3.15]{Mauldin1996}). 
The same authors proved that the upper box and packing dimensions are given by the maximum of the Hausdorff dimension of the limit set and the upper box dimension of images of any given point under the maps in the CIFS (noting that the box dimension of a countable set, unlike the Hausdorff dimension, can be strictly positive), see \cite[Theorem~2.11]{Mauldin1999}. 
They applied their results to sets of irrational numbers whose continued fraction expansions have restricted entries, as these sets are limit sets of a CIFS (see Section~\ref{ctdfracsect}). 

The key difference between the box and Hausdorff dimensions is that the box dimension is defined by covering a set with sets whose diameters are equal, whereas the definition of the Hausdorff dimension does not restrict the sizes of the covering sets. Accordingly, Falconer, Fraser and Kempton in~\cite{Falconer2020} introduced the \emph{intermediate dimensions}, defined in Section~\ref{definedimsec}, which require that the diameters of the covering sets lie in an interval $[\delta,\delta^\theta]$ for some parameter $\theta \in (0,1)$; these dimensions lie between the Hausdorff and box dimensions. In this paper we study the intermediate dimensions of limit sets of infinite iterated function systems. 
The intermediate dimensions have been studied further in~\cite{Banaji2022moran,Falconer2021,Fraser2021-1,
Burrell2021,Burrell2022brownian,Tan2020,
Burrell2022spirals,Falconer2021-2,Daw2023,
Banaji2021bedford} and have been generalised to the $\Phi$-intermediate dimensions by Banaji~\cite{Banaji2023gen} to give more refined geometric information about sets for which the intermediate dimensions are discontinuous at $\theta=0$, which by Theorem~\ref{mainint} can happen for the limit sets studied in this paper (see the discussion after Theorem~\ref{ctdfracintthm}). 
The intermediate dimensions are an example of a broader notion of `dimension interpolation' (see the survey~\cite{Fraser2021-1}), which seeks to find a geometrically natural family of dimensions which lie between two familiar notions of dimension. The family should lead to an interesting theory and satisfy properties that would be expected of reasonable notions of dimensions, and will often have some similarities with the two notions of dimension, thus potentially leading to a better understanding of the similarities and differences between the two dimensions. The other main example of dimension interpolation is the Assouad spectrum, which was introduced by Fraser and Yu~\cite{Fraser2018-2} to interpolate between the box and Assouad dimensions. 
We have investigated the Assouad spectrum of inﬁnitely generated self-conformal sets in~\cite{Banaji2022assouad}. 
 
 \subsection{Structure of paper and discussion of results}
 
 In Section~\ref{setting} we introduce notation, define limit sets, and define the notions of dimension we will be working with. We introduce and prove basic properties about the topological pressure function that we will use to obtain bounds for dimensions of the limit sets. We also define conformal iterated function systems (CIFSs) and prove geometric consequences of the definition of a CIFS that we will use in the proof of our main result, Theorem~\ref{mainint}. 
 
 In Section~\ref{mainsect} we prove Theorem~\ref{inttypeub}, which gives upper bounds for the Hausdorff, box, intermediate and $\Phi$-intermediate dimensions in terms of the topological pressure function that hold in the very general setting of arbitrary IIFSs (without any conformality assumptions or separation conditions). The proof is an induction argument, using efficient covers at larger scales to construct efficient covers at smaller scales. The natural covers we construct generally use many different scales in the allowable range. This is in contrast to sets whose intermediate dimensions have previously been calculated such as spirals~\cite{Burrell2022spirals}, sequences~\cite[Section~3.1]{Falconer2020}, and concentric spheres and topologist's sine curves~\cite{Tan2020}, where only the two extreme scales were used in the cover. 
 Recently, it has been shown that more than two scales are needed to attain the intermediate dimensions certain inhomogeneous Moran sets (see~\cite[Remark~3.12]{Banaji2022moran}) and Bedford--McMullen carpets for small values of $\theta$ (see~\cite[Corollary~2.2]{Banaji2021bedford}). In the conformal setting, Mauldin and Urbański~\cite{Mauldin1996,Mauldin1999} proved results for the Hausdorff and upper box dimensions. We use our upper bound to prove our main result: 
 \begin{theorem*}[See Theorem~\ref{mainint} for a stronger statement]
 If $F$ is the limit set of an (infinite) CIFS (defined in Definition~\ref{cifs}) and $P$ is the set of fixed points of the contractions then for all $\theta \in [0,1]$, 
 \[ \uid F = \max\{\dim_\mathrm{H} F, \uid P\}.\] 
 \end{theorem*}
 In Example~\ref{proj} we consider an example with intermediate dimensions continuous at $\theta=0$ and apply a result of Burrell, Falconer and Fraser~\cite{Burrell2021} to give an upper bound for the upper box dimension of a typical orthogonal projection.  

 In Section~\ref{ctdfracsect} we apply our results to give a formula in Theorem~\ref{ctdfracintthm} for the upper intermediate dimensions of sets of irrational numbers whose continued fraction expansions have restricted entries. The intermediate dimensions of some continued fraction sets will have a form with a phase transition that has not previously been seen in `natural' sets, see Figure~\ref{fig:holder}. We apply our results to give information about the possible H\"older coefficients of maps between continued fraction sets in Example~\ref{holderint}, observing that in some cases the intermediate dimensions for $\theta \in (0,1)$ provide better information than either the Hausdorff or box dimensions, unlike in previous cases such as elliptical polynomial spirals~\cite{Burrell2022spirals} where the box dimension gives the best information. We also show that the intermediate dimensions of continued fraction sets are continuous at $\theta=0$, and we use this fact and apply a result of Burrell~\cite{Burrell2022brownian} to show in Corollary~\ref{brownian} that the upper \emph{box} dimension of the image of any continued fraction set under index-$\alpha$ fractional Brownian motion is almost surely less than 1 if $\alpha$ is greater than the \emph{Hausdorff} dimension of the set. We also obtain similar results in Section~\ref{compsect} for sets of complex numbers which have complex continued fraction expansions with restricted entries. 
 
 In Section~\ref{genericsect} we consider the limit sets of `generic' IIFSs, in the same vein as the seminal paper~\cite{Falconer1988} where Falconer considered the generic dimension of a (finitely-generated) self-affine set by fixing a set of matrices and randomising the translates in a suitable way. We show that under certain conditions, the limit set of an IIFS with generic translates is somewhere dense, and so in particular the box and intermediate dimensions equal the ambient spatial dimension, where `generic' can have either a measure-theoretic or topological meaning. 
 This is in stark contrast to the Hausdorff dimension, which K\"aenm\"aki and  Reeve~\cite{Kaenmaki2014} showed satisfies an analogue of Falconer's affinity dimension formula for a generic IIFS of affine contractions. 

In this paper, as in~\cite{Mauldin1996,Mauldin1999}, the separation condition we assume in Definition~\ref{cifs} for a CIFS is the open set condition (OSC). Ngai and Tong~\cite{Ngai2016} and Chu and Ngai~\cite{Chu2020} study the Hausdorff, box and packing dimensions of the limit sets of IIFSs with overlaps that do not satisfy the OSC but do satisfy suitable extensions of the weak separation condition. This raises the following question. 

\begin{ques}

What can be said about the intermediate or $\Phi$-intermediate dimensions of the limit sets of infinite iterated function systems with overlaps that do not satisfy the open set condition but perhaps satisfy weaker separation conditions such as the extensions of the weak separation condition considered in~\cite{Ngai2016}? 

\end{ques}

\section{Setting and preliminaries}\label{setting}

\subsection{Notation and notions of dimension}\label{definedimsec}

We denote the natural logarithm by $\log$, the cardinality of a set using $\#$, the (Euclidean) diameter of a subset of $\Rd$ by $|\cdot|$, and $d$-dimensional Lebesgue measure on $\Rd$ by $\mathcal{L}^d$. The symbol $\N$ will denote $\{1,2,3,\ldots\}$, and $||\cdot||$ will denote either the Euclidean norm on $\Rd$ or the supremum norm of a continuous function, depending on context. 
We write 
\[ B(x,r) \coloneqq \{ \, y \in \Rd : ||y-x|| < r \, \}\ \] 
for the open ball of radius $r>0$ centred at $x \in \Rd$. %
For $F \subseteq \Rd$ let $N_r(F)$ be the smallest number of balls of radius $r$ needed to cover $F$. 
For $U \subseteq \Rd$ and $\delta >0$ let 
\[ \mathcal{S}_\delta (U) \coloneqq \{ \, x \in \Rd : \mbox{ there exists } y \in U \mbox{ such that } ||x-y|| \leq \delta \, \}\] be the closed $\delta$-neighbourhood of $U$. 
For $d \in \N$ and $r \geq 1$, we denote by $A_{d,r} \in \N$ the smallest integer such that for all $U \subset \Rd$ there exist $U_1,\ldots,U_{A_{d,r}} \subseteq \Rd$, each of diameter $|U|/r$, such that 
 \begin{equation}\label{doublingconst}
 U \subseteq \bigcup_{k=1}^{A_{d,r}} U_k.
 \end{equation}

For non-empty, bounded subsets $F$ of $\mathbb{R}^d$ with the Euclidean metric, the upper and lower box dimensions are defined by 
\[ \ubd F \coloneqq \limsup_{\delta \to 0^+} \frac{\log N_\delta(F)}{-\log \delta}; \qquad \lbd F \coloneqq \liminf_{\delta \to 0^+} \frac{\log N_\delta(F)}{-\log \delta}.\]
The packing dimension is 
\[ \dim_\mathrm{P} F \coloneqq \inf\left\{\, \sup_{i \in \N} \ubd F_i : F \subseteq \bigcup_{i=1}^\infty F_i, \mbox{ each } F_i \mbox{ non-empty and bounded}\, \right\}. \]
We can define the Hausdorff dimension without using Hausdorff measure by 
\begin{equation}\label{hausdorffdef}
\begin{aligned} \dim_\mathrm{H} F = \inf \{ \, s \geq 0 : &\mbox{ for all } \epsilon >0 \mbox{ there exists a finite or countable cover } \\*
& \{U_1,U_2,\ldots\} \mbox{ of } F \mbox{ such that } \sum_i |U_i|^s \leq \epsilon \,\},
\end{aligned}
\end{equation}%
see~\cite[Section~3.2]{Falconer2014}. In the definition of Hausdorff dimension there is no restriction on the size of the covering sets, whereas for the box dimension all the sets in the cover have the size. 
The intermediate dimensions, defined below, lie between the Hausdorff and box dimensions and were introduced in~\cite{Falconer2020}, with the restriction on the sizes of the covering sets depending on a parameter $\theta$. 
\begin{defn}\label{intdef}
For $\theta \in (0,1]$, the \emph{upper $\theta$-intermediate dimension} of a non-empty, bounded subset $F \subset \Rd$ is given by 
\begin{align*} \uid F = \inf \{ \, s \geq 0 : &\mbox{ for all } \epsilon >0 \mbox{ there exists } \delta_0 \in (0,1] \mbox{ such that for all } \\ 
&\delta \in (0,\delta_0) \mbox{ there exists a cover } \{U_1,U_2,\ldots\} \mbox{ of } F \\
&\mbox{ such that } \delta^{1/\theta} \leq |U_i| \leq \delta \mbox{ for all } i, \mbox{ and } \sum_i |U_i|^s \leq \epsilon \, \}.
\end{align*}
Similarly the \emph{lower $\theta$-intermediate dimension} of $F$ is 
\begin{align*} \lid F = \inf \{ \, s \geq 0 : &\mbox{ for all } \epsilon >0 \mbox{ and } \delta_0 \in (0,1] \mbox{ there exists } \delta \in (0,\delta_0) \\ &\mbox{ and a cover } \{U_1,U_2,\ldots\} \mbox{ of } F \mbox{ such that } \delta^{1/\theta} \leq |U_i| \leq \delta \\ &\mbox{ for all } i, \mbox{ and } \sum_i |U_i|^s \leq \epsilon \,\}.
\end{align*}
By definition, $\overline{\dim}_{\,1} = \ubd$ and $\underline{\dim}_{\,1} = \lbd$, and we define $\overline{\dim}_{\,0} = \underline{\dim}_{\,0} \coloneqq \dim_\mathrm{H}$. 
\end{defn}
For all non-empty, bounded $F \subset \Rd$, these satisfy the inequalities 
\[ 0 \leq \dim_\mathrm{H} F \leq \lid F \leq \uid F \leq \ubd F \leq d; \qquad \lid F \leq \lbd F \leq \ubd F, \]
and the maps $\theta \mapsto \uid F$ and $\theta \mapsto \lid F$ are increasing in $\theta \in [0,1]$ and continuous in $\theta \in (0,1]$ but not always at $\theta = 0$. 
In~\cite{Banaji2023gen}, motivated by the need to recover the interpolation between the Hausdorff and box dimensions when the intermediate dimensions are not continuous at $\theta = 0$, Banaji introduced the more general \emph{$\Phi$-intermediate dimensions}. These are defined by replacing the function $\delta^{1/\theta}$ in Definition~\ref{intdef} by a more general function $\Phi \colon (0,\Delta) \to \mathbb{R}$ that is monotonic and satisfies $0<\Phi(\delta) \leq \delta$ for all $\delta \in (0,\Delta)$, and $\Phi(\delta)/\delta \to 0$ as $\delta \to 0^+$; such functions are said to be \emph{admissible}. 
In this paper we will work with admissible functions that satisfy the addition mild condition that $\Phi(\delta)/\delta \to 0$ monotonically as $\delta \to 0^+$ (which is satisfied by many reasonable functions such as $\delta^{1/\theta}$ and $e^{-\delta^{-0.5}}$), and we call such functions \emph{monotonically admissible}. 

 We will use one further notion of dimension: the \emph{Assouad dimension} of any non-empty $F \subseteq \Rd$ is defined by 
\begin{multline*}
      \dim_\mathrm{A} F = \inf\left\{ \, 
      \alpha : \mbox{ there exists }C>0\mbox{ such that for all } x \in F \mbox{ and } 
    \right. \\* 
    \qquad \left. 0<r<R, \mbox{ we have } N_r(B(x,R)\cap F) \leq C(R/r)^\alpha \, \right\}. 
    \end{multline*}

\subsection{Infinite iterated function systems and pressure functions}

We will work with infinite iterated function systems, defined as in~\cite{Mauldin1996} as follows. 

\begin{defn}\label{iifs}
Let $d \in \N$ and let $X$ be a compact, connected subset of $\Rd$ with more than one point, %
equipped with the metric induced by the Euclidean norm $||\cdot ||$. We say that an \emph{infinite iterated function system (IIFS)} on $X$ is a collection of maps $S_i \colon X \to X$, $i \in I$, where $I$ is a countable index set, such that there exists $\rho \in [0,1)$ such that 
\[ ||S_i(x) - S_i(y)|| \leq \rho ||x-y|| \quad \mbox{for all } x,y \in X \mbox{ and } i \in I.\]
\end{defn}

The assumption that the maps are uniformly contracting will be important when defining the limit set. We now introduce some notation. Define $I_0 \coloneqq \{\varnothing\}$ and $I^* \coloneqq \bigcup_{i=1}^\infty I^n$. We call elements of $I^*$ \emph{finite words} and elements of $I^\N$ \emph{infinite words}. We usually denote words by the letter $w$, and we write $w = i_1 \cdots i_n$ and $w=i_1i_2\ldots$ instead of $w= (i_1,\ldots,i_n)$ and $w=(i_1,i_2,\ldots)$ respectively. We say that a word in $I^n$ has \emph{length} $n$, and an infinite word has \emph{length} $\infty$. If $w \in I^* \cup I^\N$ and $n \in \N$ does not exceed the length of $w$ then we write $w|_n \coloneqq w_1 \ldots w_n \in I^n$, and $w|_0 \coloneqq \varnothing$. If $w \in I_0 \cup I^* \cup I^\N$ and $v \in I_0 \cup I^*$ then we say that $v$ is a \emph{prefix} of $w$ if there exists $n \in \{0,1,2,\ldots\}$ such that $v = w|_n$. 
For $w \in I^n$ we define 
\[ S_w \coloneqq S_{w_1} \circ \cdots \circ S_{w_n},\]  
and we define $S_\varnothing$ to be the identity function on $X$.

Mauldin and Urbański~\cite{Mauldin1996,Mauldin1999} study the IIFSs in Definition~\ref{cifs}, where the contractions are assumed to extend to \emph{conformal} maps (see~\eqref{conformal} below for a formal definition), which means that locally they preserve angles. This assumption is crucial for Mauldin and Urbański's formulae for the Hausdorff and box dimensions of the limit set, and they will be crucial when we prove a formula for the intermediate dimensions in Section~\ref{conformalsect}. In one dimension, conformal maps are simply functions with non-vanishing H{\"o}lder continuous derivative. In two dimensions, they are holomorphic functions with non-vanishing derivative on their domain. In dimension three and higher, by a theorem of Liouville (1850) they have a very restricted form: they are M{\"o}bius transformations, so can be composed from homotheties, isometries, reflections in hyperplanes, and inversions in spheres. 

\begin{defn}\label{cifs}%
A \emph{conformal iterated function system (CIFS)} is an IIFS (as in Definition~\ref{iifs}) which satisfies the following additional properties: 
\begin{enumerate}

\item\label{osc} (Open set condition (OSC)) 
The set $X$ has non-empty interior $U \coloneqq \mathrm{Int}_{\mathbb{R}^d} X$, and $S_i(U) \subset U$ for all $i \in I$ and $S_i(U) \cap S_j(U) = \varnothing$ for all $i,j \in I$ with $i \neq j$. 

\item\label{cone} (Cone condition) $\inf_{x \in X} \inf_{r \in (0,1)} \mathcal{L}^d (B(x,r) \cap \mathrm{Int}_{\Rd} X)/r^d > 0$. 

\item\label{conformal} (Conformality) There exists an open, bounded, connected subset $V \subset \Rd$ such that $X \subset V$ and such that for each $i \in I$, $S_i$ extends to a $C^{1+\epsilon}$ diffeomorphism from $V$ to an open subset of $V$ which is \emph{conformal}, so for all $x \in V$ the differential $S_i'|_x$ exists, is non-zero, is a similarity map (so $||S_i'|_x (y)|| = ||S_i'|_x||\cdot||y||$ for all $y \in \Rd$), and is $\epsilon$-H{\"o}lder continuous in $x$. Moreover, there exists $\rho \in (0,1)$ such that for all $i \in I$ we have $||S_i'|| <\rho$, where $||\cdot||$ is the supremum norm over $V$. %

\item\label{bdp} (Bounded distortion property (BDP)) There exists $K>0$ such that for all $x,y \in V$ and $w \in I^*$ we have $||S_w'|_y|| \leq K||S_w'|_x||$. 

\end{enumerate}
\end{defn}
Mauldin and Urbański~\cite[(2.8)]{Mauldin1996} use a stronger form of the cone condition, but note on page~110 that~\eqref{cone} is sufficiently strong for their aims. Both forms of this technical condition will be satisfied if $X$ is `reasonable,' for example if $X$ is convex or has a smooth enough boundary. 

For any IIFS, since $|S_{w|_n}(X)| \leq \rho^n |X|$ by the uniform contractivity, the map 
\[ \pi \colon I^\N \to X, \qquad \pi(w) \coloneqq \bigcap_{n=1}^\infty S_{w|_n}(X) \]
is well-defined and continuous. 
We are interested in the following set, which will often be fractal in nature. 
\begin{defn}
The \emph{limit set} or \emph{attractor} of an IIFS is defined by
\[ F \coloneqq \pi(I^\N) = \bigcup_{w \in I^\N} \bigcap_{n=1}^\infty S_{w|_n}(X). \]
\end{defn}
For $w \in I^n$ define $F_w = F_{S_w} \coloneqq S_w(F)$ and $X_w = X_{S_w} \coloneqq S_w(X)$. %
Now, $F$ is clearly non-empty and satisfies the relation 
\begin{equation}\label{attractor} F = \bigcup_{i \in I} F_i. 
\end{equation}
It is the largest (by inclusion) of many sets which satisfy~\eqref{attractor}. If $I$ is finite then $F$ is compact (and is indeed the only non-empty compact set which satisfies~\eqref{attractor} by Hutchinson's application of the Banach contraction mapping theorem~\cite{Hutchinson1981}), but if $I$ is infinite then $F$ will not generally be closed. When $I$ is finite, the limit set $F$ equals the closure of the set of fixed points of all finite compositions of maps in the IFS, and also satisfies $F = \cap_{n=1}^\infty S^n(X)$ (where $S(E) \coloneqq \cup_{i \in I} S_i(E)$ for $E \subseteq X$), but when $I$ is infinite these sets may strictly contain $F$. 

We define some more quantities that will enable us to define a topological pressure function for the system. For any IIFS, for $w \in I_n$ define 
\[r_w = r_{S_w} \coloneqq \inf_{x,y \in X, x \neq y} \frac{||S_w(x)-S_w(y)||}{||x-y||};\]
\[ R_w = R_{S_w} \coloneqq \sup_{x,y \in X, x \neq y} \frac{||S_w(x)-S_w(y)||}{||x-y||},\]
noting that $0 \leq r_w \leq R_w \leq \rho$. The value $R_w$ is the smallest possible Lipschitz constant for $S_w$, and these constants are clearly submultiplicative; for any $v,w \in I^*$, we have $R_{vw} \leq R_v R_w$. 
For $n \in \N$ define $M_n \coloneqq \{ \, S_w : w \in I^n \, \}$, and for $t \in [0,\infty)$ define 
\[ \phi_n(t) \coloneqq \sum_{\sigma \in M_n} R_\sigma^t \in [0,\infty].\]
 Note that as in~\cite{Ngai2016}, we sum over $\sigma \in M_n$ instead of $w \in I^n$ so that distinct words $w$ that give rise to the same $S_w$ contribute only one term in the sum (exact overlaps are removed). 
\begin{lma}\label{fekete}
For any IIFS, for all $t \in (0,\infty)$, using the convention $\log \infty = \infty$ and $\log 0 = -\infty$, 
\[ \frac{1}{n}\log \phi_n(t) \xrightarrow[n \to \infty]{} \inf_{n \in \N} \frac{1}{n}\log \phi_n(t) \in [-\infty,\infty]. \]
\end{lma}
\begin{proof}
By the submultiplicativity of the Lipschitz constants, if $n,m \in \N$ then 
\[ \log \phi_{n+m} \leq \log \left(\sum_{\sigma \in M_n} \sum_{\tau \in M_m} (R_\sigma R_\tau)^t\right) = \log \left(\sum_{\sigma \in M_n}R_\sigma^t \sum_{\tau \in M_m} R_\tau^t\right) = \log \phi_n + \log \phi_m. \]
Therefore the sequence $(\log \phi_n)_{n=1}^\infty$ is subadditive, so the claim follows from Fekete's lemma. %
\end{proof}
In light of Lemma~\ref{fekete} we can make the following definition, which will later be used when giving bounds and formulae for the different notions of dimension for the limit sets. 
\begin{defn}
For any IIFS, define the \emph{(topological) pressure function} $\overline{P} \colon (0,\infty) \to [-\infty,\infty]$ by 
\[ \overline{P}(t) \coloneqq \lim_{n \to \infty} \frac{1}{n}\log \phi_n(t) = \inf_{n \in \N} \frac{1}{n}\log \phi_n(t).\]
\end{defn}

\begin{lma}\label{pressurestrict}
For any IIFS, $\overline{P}$ is a decreasing function, and if $0<t<s<\infty$ and $\overline{P}(t) \in \mathbb{R}$ then $\overline{P}(t)>\overline{P}(s)$. 
\end{lma}

\begin{proof}
For all $n \in \N$, for all $w \in I^n$, $R_w \leq \rho^n < 1$, so $\phi_n$ is a decreasing function. Therefore $\overline{P}$ is a decreasing function. If $0<t<s<\infty$ and $\overline{P}(t) \in \mathbb{R}$ then for all $n \in \N$ we have $\phi_n(s) \leq \rho^{(s-t)n} \phi_n(t)$ so $\frac{1}{n} \log \phi_n(s) \leq (s-t) \log \rho + \frac{1}{n} \log \phi_n(t)$ so $\overline{P}(s) \leq (s-t) \log \rho + \overline{P}(t) < \overline{P}(t)$, as required. 
\end{proof}

\begin{defn}\label{d:thetaandh}
Throughout this paper, using the convention that $\inf \varnothing = \infty$, we define the \emph{finiteness parameter} of the system 
\[ \theta_S \coloneqq \inf\{ \, t > 0 : \overline{P}(t) < \infty \, \} \in [0, \infty], \]
and the quantity 
\[ h \coloneqq \inf\{ \, t > 0 : \overline{P}(t) < 0\, \} \in [0, \infty]. \]
\end{defn}
We use the letter $h$ because we will see that it is related to the Hausdorff dimension of the limit set. 
For all $n \in \N$ and $t \in [0,\infty)$ we have 
\[ \phi_n(t) \leq \sum_{w \in I^n}R_w^t \leq \sum_{i_1\cdots i_n \in I^n} \prod_{k=1}^n R_{i_k}^t.\]
Therefore if $\phi_n(t)$ is replaced by either $\sum_{w \in I^n}R_w^t$ or $\sum_{i_1\cdots i_n \in I^n} \prod_{k=1}^n R_{i_k}^t$ in the definition of the pressure function then the resulting functions would overestimate $\overline{P}(t)$. Thus the infimal values of $t>0$ for which these new functions are negative provide upper bounds for $h$. These may be easier to compute than $h$ itself. For all $n \in \N$, the function $\frac{1}{n} \log \phi_n(t)$ can also be used to give an upper bound, which will be very good when $n$ is large, by Lemma~\ref{fekete}.

We now establish several geometric facts that hold for any CIFS and will be important when proving a formula for the intermediate dimensions of the limit set in Section~\ref{conformalsect}. The following geometric property that holds for any CIFS was proved in~\cite[page~111]{Mauldin1996}, recalling that $K$ is the constant from the bounded distortion principle. 
\begin{align*} 
(\textup{BDP.3}) \quad S_w(B(x,r)) \supseteq B(S_w(x),&K^{-1}||S_w'||r) \\*
&\textup{for all } x \in X, r \in (0,\mathrm{dist}(X,\partial V)], w \in I^*.
\end{align*}

The following lemma says that the Lipschitz constants are comparable to the norm of the derivatives of the corresponding map. 
\begin{lma}\label{diameterslemma}
For any CIFS there exists $D \geq 1$ such that for all $w \in I^*$ we have 
\[ D^{-1} ||S_w'|| \leq r_w \leq R_w \leq D||S_w'||. \]

\end{lma}

\begin{proof}

The proof is similar to the proofs of some of the consequences of the bounded distortion principle in~\cite[page~110-111]{Mauldin1996}. %
For the upper bound, note that since $X$ is compact and disjoint from the closed set $\Rd\setminus V$, we have $\mathrm{dist}(X,\Rd\setminus V) > 0$. Let $w \in I^*$. If $B$ is a ball of radius at most $\mathrm{dist}(X,\Rd\setminus V)$ centred at a point in $X$ and $x,y \in B$ then by the mean value inequality $||S_w(x)-S_w(y)|| \leq ||S_w'|| \cdot ||x-y||$. Since $X$ is compact and connected, it can be covered by a finite chain of open balls $B_1,\ldots,B_q$ centred at points in $X$ and with radii at most $\mathrm{dist}(X,\Rd\setminus V)/(2K)$ (chain in the sense that $B_i \cap B_{i+1} \neq \varnothing$ for $i=1,2,\ldots,q-1$). 
Suppose 
\[ D \geq \max\left\{q, K, \frac{K|X|}{\mathrm{dist}(X,\Rd \setminus V)}\right\}.\]
Then since $D \geq q$, the upper bound $R_w \leq D||S_w'||$ holds. 

Trivially we have $r_w \leq R_w$. For the lower bound, if $x,y \in X$ and $||x-y|| \leq \mathrm{dist}(X,\Rd \setminus V)$ then by (BDP.3) and the bijectivity of $S_w$, 
 \[||S_w(x) - S_w(y)|| \geq K^{-1}||S_w'||\cdot ||x-y|| \geq D^{-1}||S_w'||\cdot||x-y||.\]
  If $x,y \in X$ and $|X| \geq ||x-y|| > \mathrm{dist}(X,\Rd \setminus V)$ then since 
  \[ S_w(B(x,\mathrm{dist}(X,\Rd \setminus V))) \supseteq B(S_w(x),K^{-1}||S_w'||\mathrm{dist}(X,\Rd \setminus V)),\]
   we have 
 \begin{align*} 
 ||S_w(x) - S_w(y)|| \geq K^{-1}||S_w'||\mathrm{dist}(X,\Rd \setminus V) &\geq K^{-1}||S_w'||\mathrm{dist}(X,\Rd \setminus V) ||x-y||\cdot |X|^{-1} \\*
 &\geq D^{-1}||S_w'||\cdot ||x-y||.
 \end{align*}
  Therefore the lower bound $r_w \geq D^{-1} ||S_w'||$ holds, as required. 
\end{proof}

Lemma~\ref{coneimplication} essentially says that the cone condition holds not just for the set $X$ itself but also for its images under the conformal map corresponding to any finite word. The purpose of Lemma~\ref{coneimplication} is to prove Lemma~\ref{fitin}. 
\begin{lma}\label{coneimplication}
For any CIFS with $D$ as in Lemma~\ref{diameterslemma}, 
\[ \inf_{w \in I^*} \inf_{x \in S_w(X)} \inf_{r \in (0,D||S_w'||)} r^{-d} \cdot \mathcal{L}^d (B(x,r) \cap S_w(\mathrm{Int}_{\Rd} X)) > 0. \]
\end{lma}

\begin{proof}

Write $U=\mathrm{Int}_{\Rd} X$. Let $n \in \N$ and let $w \in I^n$. The idea is that given any ball centred on $S_w(X)$ whose diameter is not too large we can find a large enough ball centred on $X$ that is mapped into it under $S_w$, a uniform proportion of which intersects $U$ by the cone condition. The measure of the image of this part under $S_w$ is large enough by conformality and the BDP. 

 Consider an arbitrary point in $S_w(X)$, which we can write as $S_w(x)$ for some $x \in X$, and let $r \in (0,D||S_w'||)$. 
 By the cone condition there exists $c>0$ such that for all $x \in X$ and $r \in (0,1)$ we have $\mathcal{L}^d (B(x,r) \cap U)/r^d > c$. By the upper bound of Lemma~\ref{diameterslemma}, $S_w(B(x,r/(D||S_w'||))) \subseteq B(S_w(x),r)$. Now, $\mathcal{L}^d (B(x,r/(D||S_w'||)) \cap U) r^{-d} D^d||S_w'||^d > c$, so by the inner regularity of the Lebesgue measure, there exists a compact $C \subset B(x,r/(D||S_w'||)) \cap U$ such that $\mathcal{L}^d(C) > cr^dD^{-d}||S_w'||^{-d}$. Since $C$ is compact and disjoint from the closed set $\Rd \setminus (B(x,r/(D||S_w'||)) \cap U)$, we have $\mathrm{dist}(C,\Rd \setminus (B(x,r/(D||S_w'||)) \cap U)) > 0$. Let $n \in \N$ be large enough so that 
\[2^{-n} < \min\{\mathrm{dist}(C,\Rd \setminus (B(x,r/(D||S_w'||)) \cap U))/\sqrt{d},\mathrm{dist}(X,\Rd\setminus V)\}.\] 
Let $c_d \in (0,1)$ be the ratio of the $d$-dimensional Lebesgue measure of a $d$-dimensional ball to the $d$-dimensional Lebesgue measure of the smallest $d$-dimensional hypercube which contains it. %
Then the balls of diameter $2^{-n}$ inside each of the dyadic cubes of sidelength $2^{-n}$ which intersect $C$ form a disjoint collection of balls inside $B(x,r/(D||S_w'||)) \cap U$ whose total Lebesgue measure is greater than $cc_dr^dD^{-d}||S_w'||^{-d}$. By~\cite[(BDP.3)]{Mauldin1996}, the image of each of these balls under $S_w$ contains a ball of radius $K^{-1}||S_w'||2^{-(n+1)}$. These balls are disjoint subsets of $B(S_w(x),r) \cap S_w(U)$ whose total Lebesgue measure is greater than $c c_d r^dD^{-d}||S_w'||^{-d}||S_w'||^{d}K^{-d} = c c_d r^dD^{-d}K^{-d}$. 
Therefore 
\[ \inf_{w \in I^*} \inf_{x \in S_w(X)} \inf_{r \in (0,D||S_w'||)} \mathcal{L}^d (B(x,r) \cap S_w(\mathrm{Int}_{\Rd} X))/r^d \geq c c_d D^{-d}K^{-d} > 0, \]
as required. 
\end{proof}

Lemma~\ref{fitin} says that it is impossible for too many cylinder sets that are larger than a given size to cluster together and intersect a region that is smaller than that size. This will be useful when proving facts about dimensions, as it shows that a set of a given size cannot cover another set which intersects too many cylinders that are larger than that given size. 
Lemma~\ref{fitin} and its proof are an application of Lemma~\ref{coneimplication} and the OSC, and are similar to (2.2) in~\cite[Proposition~2.9]{Mauldin1999}. The use of the cone condition to obtain Lemma~\ref{fitin} is similar to~\cite[Theorem~4.9]{Graf1988}. 

\begin{lma}\label{fitin}%
For any CIFS there exists $M \in \N$ such that for all $z \in \Rd$ and $r>0$, if $w_1,\ldots,w_l$ are distinct words in $I^*$ such that for all $i,j \in \{1,\ldots,l\}$, $w_i$ is not a prefix of $w_j$, and for all $i \in \{1,\ldots,l\}$ we have $B(z,r) \cap S_{w_i}(X) \neq \varnothing$ and $|S_{w_i}(X)| \geq r/2$, then $l \leq M$. 
\end{lma}

\begin{proof}
Write $U=\mathrm{Int}_{\Rd} X$, and let $k_d$ be the $d$-dimensional Lebesgue measure of a ball in $\Rd$ of unit radius. 
By Lemma~\ref{coneimplication} there exists $c>0$ %
such that 
\[ \inf_{w \in I^*} \inf_{x \in S_w(X)} \inf_{r \in (0,D||S_w'||\cdot |X|)} \mathcal{L}^d (B(x,r) \cap S_w(U))/r^d > c.\] %
For each $i=1,\ldots,l$ there exists $x_i \in X$ such that $S_{w_i}(x_i) \in B(z,r)$. By the upper bound from Lemma~\ref{diameterslemma}, $r/2 \leq |S_{w_i}(X)| \leq D||S_w'||\cdot |X|$, so $\mathcal{L}^d(B(S_{w_i}(x_i),r/2) \cap S_w(U))2^d r^{-d} > c$. Since no $w_i$ is a prefix of any $w_j$, by the OSC the $l$ sets $B(S_{w_i}(x_i),r/2) \cap S_w(U)$ are disjoint subsets of $B(z,2r)$, each having $d$-dimensional Lebesgue measure at least $c r^d 2^{-d}$. Therefore $l c r^d 2^{-d} \leq \mathcal{L}^d(B(z,2r)) = k_d 2^d r^d$, so if we let $M \coloneqq k_d 2^{2d} c^{-1}$ then $l \leq M$, as required. 
\end{proof}

Lemma~\ref{fitin} shows in particular that for any CIFS, for all $n \in \N$ we have 
\[ \# \{ \, w \in I^n : B(z,r) \cap S_w(X) \neq \varnothing \textup{ and } |S_w(X)| \geq r/2 \, \} \leq M,\] 
so the family $\{ \, S_w(X) : w \in I^n \, \}$ is \emph{pointwise finite} in the sense that each element of $X$ belongs to at most finitely many elements of this family. Therefore the limit set satisfies 
\[ F = \bigcap_{n=1}^\infty \bigcup_{w \in I^n} S_w(X), \]
and so is a Borel subset of $X$ in the class $F_{\sigma \delta}$. Mauldin and Urbański noted this in~\cite[(2.5)]{Mauldin1996} and also showed that the limit set need not be in the class $G_\delta$. 
Lemma~\ref{diameterslemma} says that the Lipschitz constants are comparable to the norm of the derivative of the corresponding map; Lemma~\ref{2nddiams} uses this to show that the sizes of the corresponding cylinder sets are also comparable. 
\begin{lma}\label{2nddiams}
For any CIFS there exists $D \geq 1$ such that for all $w \in I^*$ we have 
\[ D^{-1} ||S_w'|| \leq |F_w| \leq |S_w(X)| \leq D||S_w'||. \]
\end{lma}

\begin{proof}

Lemma~\ref{fitin} shows in particular that the family $\{ \, S_i : i \in I \, \}$ is pointwise finite, so since $I$ is infinite, $F$ has positive diameter. Therefore the result follows from Lemma~\ref{diameterslemma} if we increase $D$ as required. 
\end{proof}

For any CIFS, for each $n \in \N$ define $\psi_n \colon [0,\infty) \to \mathbb{R}$ by 
\[ \psi_n(t) = \sum_{w \in I^n} ||S_w'||^t.\] Mauldin and Urbański~\cite[Section~3]{Mauldin1996} define the pressure function by 
\begin{equation}\label{MUpressure}
\lim_{n \to \infty}\frac{1}{n}\log \psi_n(t),
\end{equation} 
 showing that this limit always exists in $[-\infty,\infty]$ and proving many properties about this function. Lemma~\ref{pressureequal} shows that this coincides with our definition for the pressure function $\overline{P}(t)$, and is in particular independent of the open set $V$.

\begin{lma}\label{pressureequal}
For any CIFS, for all $t \in (0,\infty)$, 
\[ \frac{1}{n}\log \psi_n(t) \to \overline{P}(t) \textup{ as } n \to \infty. \]
\end{lma}

\begin{proof}
The OSC means that there are no exact overlaps, so $\phi_n(t) = \sum_{w \in I^n} R_w^t$ for all $n \in \N$ and $t \in (0,\infty)$. Therefore by Lemma~\ref{diameterslemma},
\[ D^{-t}\phi_n(t) \leq \psi_n(t) \leq D^t \phi_n(t). \]
Taking logarithms and dividing through by $n$ gives 
\[-\frac{1}{n}t\log D + \frac{1}{n}\log \phi_n(t) \leq \frac{1}{n}\log \psi_n(t) \leq \frac{1}{n}t\log D + \log \phi_n(t),\]
and the result follows upon taking the limit $n \to \infty$. 
\end{proof} 

One of the properties proved in~\cite[Section~3]{Mauldin1996} for a CIFS is that the finiteness parameter for the pressure function is the same as the value at which each of the finite-level approximations to the pressure function becomes finite: $\theta_S = \inf\{ \, t > 0 : \psi_n(t) < \infty \, \}$ for all $n \in \N$. %
Note that for a CIFS where each of the maps $S_i$ are similarities, which means that there exists $c_i \in (0,1)$ such that for all $x,y \in X$ we have $||S_i(x)-S_i(y)|| = c_i||x-y||$, the quantity $h$ from Definition~\ref{d:thetaandh} satisfies the simple form  
\begin{equation}\label{hausdorffsimilarity} h = \inf\left\{ \, t > 0 : \sum_{i \in I} c_i^t < 1 \, \right\}.
\end{equation}

\section{Results: dimensions of infinitely generated attractors}\label{mainsect}

\subsection{General upper bounds}

 In Theorem~\ref{inttypeub} we will provide general upper bounds for the Hausdorff, box, intermediate and $\Phi$-intermediate dimensions of the limit set of an arbitrary IIFS. We will use the following technical lemma about the $\Phi$-intermediate dimensions. 
 \begin{lma}\label{phiinvertible}
 If $\Phi \colon (0,\Delta] \to \mathbb{R}$ is monotonically admissible then there exists an \emph{invertible} monotonically admissible function $\Phi_1 \colon (0,\Delta] \to (0,\Phi(\Delta)]$ such that for all non-empty, bounded subsets $F \subset \Rd$ we have $\upd F = \overline{\dim}^{\Phi_1} F$ and $\lpd F = \underline{\dim}^{\Phi_1} F$. 
 \end{lma}
 
 \begin{proof}
 Define $\Phi_1$ by $\Phi_1(\Delta/2^n) \coloneqq \Phi(\Delta/2^n)$ and $\Phi_1$ linear on $[\Delta/2^{n+1},\Delta/2^n]$ for $n=0,1,2,\ldots$. Then clearly $\Phi_1$ is invertible and satisfies $\Phi_1(\delta) \leq \delta$ for all $\delta \in (0,\Delta]$ and $\Phi_1(\delta)/\delta \searrow 0$ monotonically as $\delta \to 0^+$. Moreover $\Phi(\delta/2) \leq \Phi_1(\delta) \leq \Phi(2\delta)$ for all $\delta \in (0,\Delta/2]$, so by~\cite[Corollary~3.8~(ii)]{Banaji2023gen}, $\upd F = \overline{\dim}^{\Phi_1} F$ and $\lpd F = \underline{\dim}^{\Phi_1} F$ for all non-empty, bounded $F \subset \Rd$.
 \end{proof}
 
\begin{thm}\label{inttypeub}
For any IIFS with limit set $F$ and notation as above, 
\begin{enumerate}
\item\label{hausdorffub} 
$\dim_\mathrm{H} F \leq h$
\item\label{boxub} $\ubd F \leq \max\{h,\lim_{n \to \infty} \inf\{ \, \ubd P : P \subseteq X \textup{ and } P \cap S_w(X) \neq \varnothing \, \forall w \in I^n \, \} \}$
\item\label{intub} For all $\theta \in [0,1]$ we have 
\[ \uid F \leq \max\{h,\lim_{n \to \infty} \inf\{ \, \uid  P : P \subseteq X \textup{ and } P \cap S_w(X) \neq \varnothing \, \forall w \in I^n \, \} \}\]
\item\label{phiub} If $\Phi$ is monotonically admissible then 
\[ \upd F \leq \max\{h,\lim_{n \to \infty} \inf\{ \, \upd  P : P \subseteq X \textup{ and } P \cap S_w(X) \neq \varnothing \, \forall w \in I^n \, \} \}\]
\end{enumerate}

\end{thm}%

Since $\dim_\mathrm{P} F \leq \ubd F$ always holds,~\eqref{boxub} also gives an upper bound for the packing dimension. 
Note that the above upper bounds hold even when there are overlaps of cylinders, and for contractions which are not differentiable and do not satisfy any bi-Lipschitz or bounded distortion condition. However, in some such cases $h$ can overestimate $\dim_\mathrm{H} F$ significantly, and may even be infinite. 

\begin{proof} 

All the bounds are trivial if $h=\infty$, so assume $h<\infty$.

\eqref{hausdorffub} The proof is similar to the proof of the first part of~\cite[Theorem~3.15]{Mauldin1996}. Let $s>h$. By Lemma~\ref{pressurestrict} and the definition of $h$, $\overline{P}(s)<0$. Therefore there exists $N \in \N$ such that for all $n \geq N$ we have $\frac{1}{n}\log \phi_n(s) < \overline{P}(s)/2$, so $\phi_n(s) < e^{n\overline{P}(s)/2}$. Therefore 
\[ \sum_{\sigma \in M_n} |\sigma(X)|^s \leq |X|^s \phi_n(s) <  |X|^s e^{n\overline{P}(s)/2} \xrightarrow[n \to \infty]{} 0.\]
 But $\{\, \sigma(X) : \sigma \in M_n \, \}$ forms a $|X|\rho^n$-cover of $F$, and $|X|\rho^n \to 0$ as $n \to \infty$, so this means that the $s$-dimensional Hausdorff measure of $F$ is 0. Thus $\dim_\mathrm{H} F \leq s$. Letting $s \to h^+$ gives $\dim_\mathrm{H} F \leq h$, as required. 

\eqref{boxub} follows from the case $\theta=1$ of~\eqref{intub}. 

\eqref{intub} The proof is motivated by the proof of~\cite[Lemma~2.8]{Mauldin1999}, which gives a result for the box dimension in the less general setting of a CIFS. 
We will consider $\delta \in \left(\frac{1}{n+1},\frac{1}{n}\right]$ and induct on $n$. The idea is that if we fix a large enough $q \in \N$, the level-$q$ cylinders with size $\lesssim \delta$ can be covered efficiently using a cover of a set $P$ corresponding to level~$q$, and the cylinders with size $\gtrsim \delta$ can be covered efficiently using images of efficient covers of $F$ with larger diameters that are assumed to exist by the inductive hypothesis, and the fact that $\overline{P}(s) < 0$ if $s > h$.  

By definition $\overline{\dim}_{\,0} = \dim_\mathrm{H}$, so since we can take $P$ to be a countable set (with Hausdorff dimension 0) for all $n \in \N$, the case $\theta = 0$ follows from~\eqref{hausdorffub}. Henceforth suppose $\theta \in (0,1]$. Let 
\[ s > \max\{h,\lim_{n \to \infty} \inf\{ \, \uid P : P \subseteq X \textup{ and } P \cap S_w(X) \neq \varnothing \, \forall w \in I^n \, \} \}. \]
Since $s>h$, we have $\overline{P}(s)<0$. Therefore there exists $Q \in \N$ such that for all $q \geq Q$ we have $\frac{1}{q}\log \phi_q(s) < \overline{P}(s)/2$. Fix $q \geq Q$ large enough such that $\phi_q(s) \leq 1/2$ and $\rho^q < 1/4$, so $R_w < 1/4$ for all words $w$ of length at least $q$. By the definition of $s$, increasing $q$ further if necessary, we may assume there exists a subset $P_q \subseteq X$ such that $P_q \cap S_w(X) \neq \varnothing$ for all $w \in I^q$, and $\uid P_q < s$. 
Therefore there exists $A>0$ such that for all $\delta \in (0,1]$ there exists a cover $\{V_j\}$ of $P_q$ such that $\delta \leq |V_j| \leq \delta^\theta$ for all $j$, and $\sum_j |V_j|^s \leq A$. 
Let $A_{d,1+2|X|}$ be as in~\eqref{doublingconst}. 
Fix any $B>\frac{A_{d,1+2|X|} A}{1-\phi_q(s)}$ large enough so that for all $\delta \in (1/2,1]$ there exists a cover $\{U_i^\delta\}_i$ of $F$ such that $\delta \leq |U_i| \leq \delta^\theta$ for all $i$, and $\sum_i |U_i|^s \leq B$. It suffices to show that $\uid F \leq s$, which follows from the following claim. 

\textbf{Claim:} For all $n \in \N$, for all $\delta \in \left(\frac{1}{n+1},\frac{1}{n}\right]$ there exists a cover $\{U_i^\delta\}_i$ of $F$ such that $\delta \leq |U_i| \leq \delta^\theta$ for all $i$, and $\sum_i |U_i|^s \leq B$. 

\textbf{Proof of claim:} We prove the claim by induction on $n$. The claim holds for $n=1$ by the definition of $B$. Let $n \in \N$, $n >1$, and assume the claim holds for $1,2,\ldots,n-1$. Let $\delta \in \left(\frac{1}{n+1},\frac{1}{n}\right]$. By the definition of $A$ there exists a cover $\{V_j\}$ of $P_q$ such that $\delta \leq |V_j| \leq \delta^\theta$ for all $j$, and $\sum_j |V_j|^t \leq A$. 
By the definition of $A_{d,1+2|X|}$, for all $j$ there exist $V_{j,1},\ldots,V_{j,A_{d,1+2|X|}} \subseteq \Rd$, each of diameter 
\[\max\{\delta,|\mathcal{S}_{|X|\delta}(V_j)|/(1+2|X|)\},\]
 such that 
 \[\mathcal{S}_{|X|\delta}(V_j) \subseteq \bigcup_{k=1}^{A_{d,1+2|X|}} V_{j,k}.\]
  By the triangle inequality, 
\[|\mathcal{S}_{\delta |X|} (V_j)| \leq |V_j| + 2|X|\delta \leq (1+2|X|)|V_j| \leq (1+2|X|)\delta^\theta,\]
so $\delta \leq |V_{j,k}| \leq \delta^\theta$ and $|V_{j,k}| \leq |V_j|$ for all $j,k$. 
Recalling that $M_q$ is the set of maps corresponding to words of length $q$, let $C_{\delta} \coloneqq \{ \, \tau \in M_q : |F_\sigma| \leq |X|\delta \, \}$. 
Since $\{V_j\}$ covers $P_q$, $\{\mathcal{S}_{|X|\delta}(V_j)\}$ covers $\cup_{\tau \in C_\delta} F_\tau$, so $\cup_{k=1}^{A_{d,1+2|X|}} V_{j,k}$ covers $\cup_{\tau \in C_\delta} F_\tau$. 

Now suppose $\sigma \in M_q \setminus C_\delta$, so $|X|\delta < |F_\sigma| \leq |X|R_\sigma$, so $\delta/R_\sigma < 1$, and since $R_\sigma <1/4$ we have 
\[ \frac{\delta}{R_\sigma} \geq \frac{1}{(n+1)R_\sigma} > \frac{4}{n+1} > \frac{1}{n}.\] Therefore by the inductive assumption there exists a cover $\{U_i^{\delta/R_\sigma}\}$ of $F$ such that $\delta/R_\sigma \leq |U_i^{\delta/R_\sigma}| \leq (\delta/R_\sigma)^\theta$ for all $i$ and $\sum_i |U_i^{\delta/R_\sigma}|^s \leq B$. 
For each $i$, let $W_{\sigma,i}$ be a set with diameter 
\begin{equation}\label{wsigma} 
|W_{\sigma,i}| = \max\{|S_\sigma(U_i^{\delta/R_\sigma})|,\delta\}
\end{equation}
 such that $S_\sigma(U_i^{\delta/R_\sigma}) \subseteq W_{\sigma,i}$. Since $\{S_\sigma (U_i^{\delta/R_\sigma})\}_i$ covers $F_\sigma$, also $\{W_{\sigma,i}\}_i$ covers $F_\sigma$. By the definition of $R_\sigma$, for all $j$ we have $|S_\sigma(U_i^{\delta/R_\sigma})| \leq R_\sigma|U_i^{\delta/R_\sigma}|$, and also $\delta = R_\sigma \delta/R_\sigma \leq R_\sigma |U_i^{\delta/R_\sigma}|$, so by~\eqref{wsigma} we have 
\begin{equation}\label{intkeybound} \delta \leq |W_{\sigma,i}| \leq R_\sigma|U_i^{\delta/R_\sigma}| \leq R_\sigma (\delta/R_\sigma)^\theta \leq \delta^\theta.
\end{equation}
 The last inequality (which is crucial to the argument) holds since $R_\sigma < 1$. 
 Now, $\{V_{j,k}\} \cup \{W_{\sigma,i}\}$ is a cover of $F$ and the diameter of each of these sets lies in the interval $[\delta,\delta^\theta]$. 
Moreover, since $|V_{j,k}| \leq |V_j|$ and by~\eqref{intkeybound}, 
\begin{align*}
\sum_j \sum_{k=1}^{A_{d,1+2|X|}} |V_{j,k}|^s + \sum_{\sigma \in M_q \setminus C_\delta} \sum_i |W_{\sigma,i}|^s 
&\leq A_{d,1+2|X|} \sum_j |V_j|^s + \sum_{\sigma \in M_q \setminus C_\delta} R_\sigma^s \sum_i |U_i^{\delta/R_\sigma}|^s \\
&\leq A_{d,1+2|X|} A + B\phi_q(s) \\
&\leq B,
\end{align*}
so the claim holds by induction.

\eqref{phiub} By Lemma~\ref{phiinvertible} we may assume without loss of generality that $\Phi$ is invertible. Then~\eqref{phiub} holds by almost exactly the same proof as~\eqref{intub}, with $\uid$, $\delta^\theta$ and $(\delta/R_w)^\theta$ replaced by $\upd$, $\Phi^{-1}(\delta)$ and $\Phi^{-1}(\delta/R_w)$ respectively throughout. In place of~\eqref{intkeybound}, the key inequality $R_w \Phi^{-1}(\delta/R_w) \leq \Phi^{-1}(\delta)$ holds since $\Phi(\delta)/\delta \searrow 0$ monotonically as $\delta \to 0^+$ by assumption. 
\end{proof}

Elements of $M_q \setminus C_\delta$ can have many different contraction ratios, and we see from~\eqref{intkeybound} that many different scales in the interval $[\delta,\delta^\theta]$ will generally be used in the construction of the cover.

\subsection{Precise formulae for conformal iterated function systems}\label{conformalsect}

In order to use the upper bounds in Theorem~\ref{inttypeub} to prove a simpler formula for intermediate dimensions of the limit set of a CIFS in Theorem~\ref{mainint} as the maximum of the Hausdorff dimension and the intermediate dimensions of the fixed points, we need further lemmas. 
The following lemma and proof are similar to~\cite[Proposition~2.9]{Mauldin1999} for the box dimension. 

\begin{lma}\label{samewithinlevel}%
Consider a CIFS and fix $n \in \N$. If $P$ and $Q$ are both subsets of $\cup_{w \in I^n} S_w(X)$ which intersect $S_w(X)$ in exactly one point for each $w \in I^n$ then $\uid P = \uid Q$ for all $\theta \in [0,1]$, and $\upd P = \upd Q$ for all monotonically admissible functions $\Phi$.  
The same holds with $\overline{\dim}$ replaced by $\underline{\dim}$ throughout. 
\end{lma}

\begin{proof} 
The idea is to use an efficient cover of $Q$ at scale $\delta$ to construct an efficient cover of $P$ at scale $\delta$. Elements of $P$ in cylinders of size $\lesssim \delta$ can be covered using the cover of the element of $Q$ in the same cylinder, and the conditions of a CIFS (via Lemma~\ref{fitin}) mean that any element of the cover can intersect only a bounded number of cylinders that are larger than the covering set in question. 

Since $\{ \, S_w : w \in I^n \, \}$ forms a CIFS with the same limit set, we may henceforth assume without loss of generality that $n=1$. 
Let $A_{d,3}$ be as in~\eqref{doublingconst} and let $M$ be as in Lemma~\ref{fitin}. If $\theta = 0$ then $\uid P = \uid Q = 0$ because $P$ and $Q$ are countable, so henceforth assume that $\theta \in (0,1]$. 

\textbf{Claim:} If $\{U_j\}$ is a cover of $Q$ such that $\delta \leq |U_j| \leq \delta^\theta$ for all $j$ then there exists a cover $\{V_m\}$ of $P$ such that $\delta \leq |V_m| \leq \delta^\theta$ for all $m$, and $\sum_m |V_m|^s \leq (A_{d,3} + M)\sum_j |U_j|^s$ for all $s\geq 0$.

\textbf{Proof of claim:} For each $j$, if $i \in I$ is such that $|S_i(X)| \leq |U_j|$ and $S_i(X) \cap U_j \neq \varnothing$, then 
\[ S_i(X) \subseteq \mathcal{S}_{|U_j|}(U_j) \subseteq \bigcup_{l=1}^{A_{d,3}} \mathcal{S}_{|U_j|}(U_j)_l,\]
 where $\mathcal{S}_{|U_j|}(U_j)$ is the neighbourhood set, which has diameter $3|U_j|$. By Lemma~\ref{fitin} there exist $i_1, \ldots, i_M \in I$, not necessarily distinct, such that $S_{i_k}(X) \cap U_j \neq \varnothing$ for $k=1,\ldots,M$, and such that if $i \in I \setminus \{i_1,\ldots,i_M\}$ and $|S_i(X)| > |U_j|$ then $S_i(X) \cap U_j = \varnothing$. If $k=1,\ldots,M$ then we can cover the single element of $P \cap S_{i_k}(X)$ by a ball $B_{j,k}$ of diameter $|U_j|$. Since $\{U_j\}$ covers $Q$, 
\[P \subseteq \bigcup_j \left( \bigcup_{l=1}^{A_{d,3}} \mathcal{S}_{|U_j|}(U_j)_l \cup \bigcup_{k=1}^M B_{j.k} \right).\]
 Each element of this cover of $P$ has diameter in the interval $[\delta,\delta^\theta]$ by construction. Moreover, 
 \[ \sum_j \left(\sum_{l=1}^{A_{d,3}} |\mathcal{S}_{|U_j|}(U_j)_l|^s + \sum_{k=1}^M |B_{j,k}|^s\right) = (A_{d,3} + M)\sum_j |U_j|^s, \]
 proving the claim. 
 
 The claim shows that $\uid P \leq \uid Q$ and $\lid P \leq \lid Q$. The reverse inequalities hold by symmetry, so $\uid P = \uid Q$ and $\lid P = \lid Q$, as required. 
 If $\Phi$ is a monotonically admissible function then by Lemma~\ref{phiinvertible} we may assume without loss of generality that $\Phi$ is invertible. Then the same proof works with $\delta^\theta$, $\uid$ and $\lid$ replaced by $\Phi^{-1}(\delta)$, $\upd$ and $\lpd$ respectively throughout. 
 \end{proof}

Lemma~\ref{changelevel} shows that the upper intermediate dimensions of a set of points corresponding to the $n$th level cylinders are either all bounded above by the finiteness parameter, and hence the Hausdorff dimension of the limit set, or they all equal the upper intermediate dimensions of the 1st-level fixed points. 
We will combine this lemma with the upper bounds in Theorem~\ref{inttypeub} (which considers arbitrarily deep levels) to prove that the dimensions in fact depend only on the 1st-level fixed points (and the Hausdorff dimension). 
Mauldin and Urbański prove in~\cite[Lemma~2.10]{Mauldin1999} that the upper box dimension of the 1st-level iterates of a given point is greater than or equal to the finiteness parameter $\theta_S$, and deduce that it equals the box dimension of the $n$th level iterates for all $n \in \N$. The intermediate dimensions, on the other hand, will \emph{not} always exceed the finiteness parameter, so we cannot make the same conclusion for the intermediate dimensions in Lemma~\ref{changelevel}. 

 \begin{lma}\label{changelevel}
 Consider a CIFS, and suppose that for each $n \in \N$, $P_n \subseteq \cup_{w \in I^n} S_w(X)$ intersects $S_w(X)$ in exactly one point for each $w \in I^n$. Then 
 
 \begin{enumerate}
 \item\label{changeleveluid} for all $\theta \in [0,1]$, either $\uid P_n \leq \theta_S \leq h$ for all $n \in \N$ or $\uid P_n = \uid P_1$ for all $n \in \N$. 
 \item\label{changelevelupd} If $\Phi$ is monotonically admissible then either $\upd P_n \leq \theta_S \leq h$ for all $n \in \N$ or $\upd P_n = \upd P_1$ for all $n \in \N$. 
 \end{enumerate}
 \end{lma}%

 \begin{proof}
 \eqref{changeleveluid} This is true for the Hausdorff dimension because each $P_n$ is countable, so henceforth suppose $\theta \in (0,1]$. Since $\uid P_n$ does not depend on the particular set $P_n$ by Lemma~\ref{samewithinlevel}, and since $\uid$ is monotonic for subsets, we have $\uid P_1 \leq \uid P_2 \leq \ldots$. 
 By Lemma~\ref{samewithinlevel}, we can henceforth fix $x \in X$ and assume without loss of generality that $P_n \coloneqq \{ \, S_w(x) : w \in I^n \, \}$ for all $n \in \N$. 
 It suffices to prove that $\uid P_n \leq s$ for all $n \in \N$, which follows from the following claim. 
 
 \textbf{Claim:} For all $s>\max\{\theta_S, \uid P_1\}$, for all $n \in \N$ there exists $B_n \in (0,\infty)$ such that for all $\delta \in (0,1]$ there exists a cover $\{U_j^{\delta,n}\}_j$ of $P_n$ such that $\delta \leq |U_j^{\delta,n}| \leq \delta^\theta$ for all $i$ and $\sum_j |U_j^{\delta,n}|^s \leq B_n$. 
 
 \textbf{Proof of claim:} Fix $s>\max\{\theta_S, \uid P_1\}$. We prove the claim by induction on $n$. Suppose $n>1$ and assume the claim holds for $1,2,\ldots,n-1$. 
 Let $\delta \in (0,1]$. 
 By the definition of $A_{d,1+2|X|}$ in~\eqref{doublingconst}, for all $j$ there exist $U_{j,1}^{\delta,n-1},\ldots,U_{j,A_{d,1+2|X|}}^{\delta,n-1} \subseteq \Rd$, each of diameter 
 \[ \max\left\{\delta,\frac{|\mathcal{S}_{|X|\delta}(U_j^{\delta,n-1})|}{1+2|X|}\right\},\]
  such that 
  \[\mathcal{S}_{|X|\delta}(U_j^{\delta,n-1}) \subseteq \cup_{k=1}^{A_{d,1+2|X|}} U_{j,k}^{\delta,n-1}.\]
   By the triangle inequality, 
\[|\mathcal{S}_{\delta |X|} (U_j^{\delta,n-1})| \leq |U_j^{\delta,n-1}| + 2|X|\delta \leq (1+2|X|)|U_j^{\delta,n-1}| \leq (1+2|X|)\delta^\theta.\]
Therefore $\delta \leq |U_{j,k}^{\delta,n-1}| \leq \delta^\theta$ and $|U_{j,k}^{\delta,n-1}| \leq |U_j^{\delta,n-1}|$ for all $j,k$.
 
 Let $C_{\delta} \coloneqq \{ \, w \in I^{n-1} : |X_w| \leq |X|\delta \, \}$. If $w \in C_\delta$ then there exists $p_w \in S_w(X) \cap P_{n-1}$, and there exists $j$ such that $p_w \in U_j^{\delta,n-1}$, so the neighbourhood set $\mathcal{S}_{|X|\delta} (U_j^{\delta,n-1})$ covers $S_w(X)$. Thus 
 \begin{equation}\label{changelevelinclusion} P_n \cap S_w(X) \subseteq S_w(X) \subseteq \mathcal{S}_{|X|\delta} (U_j^{\delta,n-1}) \subseteq \bigcup_{k=1}^{A_{d,1+2|X|}} U_{j,k}^{\delta,n-1}. 
 \end{equation}
 If, on the other hand, $w \in I^{n-1}\setminus C_\delta$, then $|X|\delta < |X_w| \leq |X|R_w$ and $\delta/R_w < 1$. Consider the cover $\{U_l^{\delta/R_w,1}\}_l$ of $P_1$ whose existence is guaranteed by the base case $n=1$. 
  For each $l$, let $W_{w,l}$ be a set with diameter $|W_{w,l}| = \max\{|S_w(U_l^{\delta/R_w,1} \cap X)|,\delta\}$ such that $S_w(U_l^{\delta/R_w,1} \cap X) \subseteq W_{w,l}$. Since $P_n \coloneqq \{ \, S_w(x) : w \in I^n \, \}$, $\{S_w(U_l^{\delta/R_w,1} \cap X)\}_l$ covers $P_n \cap S_w(X)$, so $\{W_{w,l}\}_l$ covers $P_n \cap S_w(X)$. By the definition of $R_w$, for all $l$ we have $|S_w(U_l^{\delta/R_w,1} \cap X)| \leq R_w |U_l^{\delta/R_w,1} \cap X| \leq R_w |U_l^{\delta/R_w,1}|$, %
  and also $\delta = R_w\delta/R_w \leq R_w |U_l^{\delta/R_w,1}|$, so 
 \begin{equation}\label{intlemmakeybound}
 \delta \leq |W_{w,l}| \leq R_w |U_l^{\delta/R_w,1}| \leq R_w (\delta/R_w)^\theta \leq \delta^\theta. 
 \end{equation}
 Now, $\{U_{j,k}^{\delta,n-1}\} \cup \{W_{w,l}\}$ covers $P_n$ and the diameter of each element of this cover lies in the interval $[\delta,\delta^\theta]$. Moreover, since $|U_{j,k}^{\delta,n-1}| \leq |U_j^{\delta,n-1}|$ for all $j,k$, and by~\eqref{intlemmakeybound}, we have 
 \begin{align*}
  \sum_j \sum_{k=1}^{A_{d,1+2|X|}}|U_{j,k}^{\delta,n-1}|^s &+ \sum_{w \in I^n \setminus C_\delta} \sum_l |W_{w,l}|^s \\
  &\leq A_{d,1+2|X|} \sum_j |U_j^{\delta,n-1}|^s + \sum_{w \in I^n \setminus C_\delta} R_w \sum_l |U_l^{\delta/R_w,1}|^s \\
 &\leq A_{d,1+2|X|} B_{n-1} + B_1 \phi_n(s).
 \end{align*}
 Therefore letting $B_n \coloneqq A_{d,1+2|X|} B_{n-1} + B_1 \phi_n(s)$, since $s>\theta_S$, $\phi_n(s) < \infty$, so $B_n < \infty$, and the claim holds by induction. 
 
 \eqref{changelevelupd} By Lemma~\ref{phiinvertible} we may assume that $\Phi$ is invertible. Then~\eqref{changelevelupd} holds by the same proof as~\eqref{changeleveluid}, with $\uid$, $\delta^\theta$ and $(\delta/R_w)^\theta$ replaced by $\upd$, $\Phi^{-1}(\delta)$ and $\Phi^{-1}(\delta/R_w)$ respectively. In place of~\eqref{intlemmakeybound}, $R_w \Phi^{-1}(\delta/R_w) \leq \Phi^{-1}(\delta)$ holds since $\Phi(\delta)/\delta \searrow 0$ monotonically as $\delta \to 0^+$ by assumption. 
 \end{proof}

Mauldin and Urbański~\cite[Theorem~3.15]{Mauldin1996} show that the Hausdorff dimension of the limit set $F$ of a CIFS is $h$. 
In fact, this is true even if the cone condition~\eqref{cone} is not assumed, but in this paper we do use the cone condition in the proof of Lemma~\ref{coneimplication} (and hence Lemmas~\ref{fitin} and~\ref{samewithinlevel}). 
We now use the fact that $h = \dim_{\mathrm H} F$, together with the upper bounds in Theorem~\ref{inttypeub} and Lemmas~\ref{changelevel} and \ref{samewithinlevel}, to prove our main result, which gives the following simple formulae for other dimensions of the limit set as the maximum of the Hausdorff dimension of the limit set and the corresponding dimension of any set $P$ which consists of exactly one point of each cylinder set. 

\begin{thm}\label{mainint}
For any CIFS with limit set $F$ and notation as above, if $P$ is any subset of $\cup_{i \in I} S_i(X)$ which intersects $S_i(X)$ in exactly one point for each $i \in I$, then 
\begin{enumerate}
\item\label{boxeq} $\ubd F = \dim_\mathrm{P} F = \max\{h,\ubd P\}$ (very similar to Mauldin and Urbański~\cite[Theorem~2.11]{Mauldin1999} but with a more general condition on $P$),
\item\label{inteq} $\uid F = \max\{h,\uid P\}$ for all $\theta \in [0,1]$,
\item\label{phieq} $\upd F = \max\{h,\upd P\}$ if $\Phi$ is monotonically admissible.
\end{enumerate}
\end{thm}

Examples for the set $P$ include $\{ \, S_i(x) : i \in I \, \}$ for any given $x \in X$, as in~\cite{Mauldin1999}, or the set of fixed points in $X$ of the contractions $S_i$. 

\begin{proof}

\eqref{boxeq} follows from the case $\theta=1$ of~\eqref{inteq} and the fact that $\ubd F = \dim_\mathrm{P} F$ by~\cite[Theorem~3.1]{Mauldin1996}. 

\eqref{inteq} For each $n \in \N$ let $P_n \coloneqq \{ \, x \in X : x = S_w(x) \mbox{ for some } w \in I^n \, \}$, so $P_n \subseteq F \subseteq \cup_{w \in I^n} S_w(X)$. 
Then by Lemmas~\ref{samewithinlevel} and \ref{changelevel}~\eqref{changeleveluid}, $\uid P_n \leq \max\{h,\uid P\}$ for all $n \in \N$. Therefore by Theorem~\ref{inttypeub}~\eqref{intub}, 
\[ \uid F \leq \max\{ h, \lim_{n \to \infty} \uid P_n \} \leq \max\{ h, \max\{h,\uid P \} \} = \max\{h,\uid P\}.\]
 But $P_n \subseteq F$ so since $\uid$ is monotonic for subsets, by Lemma~\ref{samewithinlevel} we have $\uid P = \uid P_1 \leq \uid F$, and by~\cite[Theorem~3.15]{Mauldin1996}, $h = \dim_\mathrm{H} F \leq \uid F$, so $\max\{h,\uid P\} \leq \uid F$. Therefore $\uid F = \max\{h,\uid P\}$, as required. 

\eqref{phieq} is similar to~\eqref{inteq}. 
\end{proof}

As a consequence we have the following bounds for the lower versions of the dimensions in Theorem~\ref{mainint}. 

\begin{cor}\label{lowerintcor}
For any CIFS with limit set $F$, if $P$ is any subset of $\cup_{i \in I} S_i(X)$ which intersects $S_i(X)$ in exactly one point for each $i \in I$, then 
\begin{enumerate}
\item\label{lowerbox} $\max\{h,\lbd P\} \leq \lbd F \leq \max\{h,\ubd P\}$
\item\label{lowerint} $\max\{h,\lid P\} \leq \lid F \leq \max\{h,\uid P\}$ for all $\theta \in [0,1]$ 
\item\label{lowerphi} $\max\{h,\lpd P\} \leq \lpd F \leq \max\{h,\upd P\}$ if $\Phi$ is monotonically admissible. 
\end{enumerate}
\end{cor} 

\begin{proof}
We prove~\eqref{lowerint};~\eqref{lowerbox} and~\eqref{lowerphi} are similar. By Theorem~\ref{mainint}~\eqref{inteq} $\lid F \leq \uid F = \max\{h,\uid P\}$. If $P_1$ is the set of fixed points in $X$ of the maps $\{S_i\}_{i \in I}$ then by Lemma~\ref{samewithinlevel} we have $\lid P = \lid P_1 \leq \lid F$, and by~\cite[Theorem~3.15]{Mauldin1996}, $h = \dim_\mathrm{H} F \leq \lid F$, so $\max\{h,\lid P\} \leq \lid F$, as required. 
\end{proof} 

\begin{ques}
Are the bounds in Corollary~\ref{lowerintcor} sharp or can they be improved in general? 
\end{ques}

We remark that for infinite systems where finitely many of the maps have a parabolic fixed point, we can obtain information about the intermediate dimensions of the limit set by applying Theorem~\ref{mainint} to an `induced’ uniformly contracting CIFS (see \cite[Section~6]{Banaji2022assouad}). 

\subsection{An example and a first application}

In Proposition~\ref{lattice} we compute the intermediate dimensions of the inversion of the lattice $\{ 1^p,2^p,3^p,\ldots \}^d$ in the unit $d$-sphere in $\Rd$. The case $d=1$ is just the polynomial sequence sets $\{ 1^{-p},2^{-p},3^{-p},\ldots \}$ whose intermediate dimensions were calculated in~\cite[Proposition~3.1]{Falconer2020}. 

\begin{prop}\label{lattice}
For $d \in \N$ and $p \in (0,\infty)$ define \[ G_{p,d} \coloneqq \{ \, x/||x||^2 : x \in \{ 1^p,2^p,3^p,\ldots \}^d \, \}. \]
 Then for all $\theta \in [0,1]$, 
\[ \dim_{\, \theta} G_{p,d} = \frac{d\theta}{p + \theta}. \]
In particular the intermediate dimensions are continuous at $\theta=0$. 
\end{prop} 

\begin{proof}
We begin with the upper bound. Let $\theta \in (0,1]$. 
For $\delta \in (0,1/10)$ let $n \coloneqq \lceil \delta^{-\theta/(p+\theta)} \rceil$. 
We form a cover $\mathcal{U}$ by covering each point in $\{ \, x/||x||^2 : x \in \{ 1^p,2^p,\ldots n^p\}^d \, \}$ with a ball of diameter $\delta$, and covering $[0,n^{-p}]^d$ with $\lesssim (n^{-p}/\delta^\theta + 1)^d$ sets of diameter $\delta^\theta$, where $\lesssim$ means up to a multiplicative constant independent of $\delta$ and $n$. Then for $s > \frac{d\theta}{p + \theta}$, 
\begin{align*}
 \sum_{U \in \mathcal{U}} |U|^s &\lesssim (n^{-p}/\delta^\theta + 1)^d \delta^{\theta s} + n^d \delta^s \\
 &\lesssim \delta^{dp\theta/(p+\theta)}\delta^{-d \theta}\delta^{\theta s} + \delta^{\theta s} + \delta^{-d\theta/(p+\theta)}\delta^s \\
 &\lesssim \delta^{\theta(s-d\theta/(p+\theta))} + \delta^{s-d\theta/(p+\theta)} \\
 &\lesssim 1,
 \end{align*}
proving $\uid G_{p,d} \leq s$. 

For the lower bound, for $\delta \in (0,1/10)$, let $m \coloneqq \lceil \delta^{-1/(p+1)} \rceil$. 
A direct geometric argument shows that $\{ \, x/||x||^2 : x \in \{ 1^p,2^p,\dotsc, m^p\}^d \, \}$ is a $\gtrsim \delta$-separated set, so if $0 < p < 1$ then 
\[ N_\delta (G_{p,d} ) \geq N_\delta (\{ \, x/||x||^2 : x \in \{ 1^p,2^p,\dotsc, m^p\}^d \, \}) \gtrsim m^d \geq \delta^{-d/(p+1)}.\]
A geometric argument shows that for each $\eta > 0$, 
\begin{equation*}
\sup_{x \in [0,m^{-p}]^d} \inf_{y \in G_{p,d}} ||x-y|| \lesssim \delta^{1-\eta},
\end{equation*} 
so if $p \geq 1$ then 
\[ N_\delta (G_{p,d}) \geq N_\delta (G_{p,d} \cap [0,m^{-p}]^d ) \gtrsim \delta^{d\eta} (m^{-p}/\delta)^d \gtrsim \delta^{-(d/(p+1) - d\eta)}. \]
Therefore $\lbd G_{p,d} \geq d/(p+1)$ for all $p \in (0,\infty)$. 
Moreover, for all $\eta > 0$ sufficiently small, 
\[ N_{\delta^{1-\eta}} ([0,m^{-p}] \cap G_{p,d}) \approx \left(\frac{m^{-p}}{\delta^{1-\eta}}\right)^d, \]
so $\dim_{\mathrm A} G_{p,d} = d$. 
By the general lower bound~\cite[Corollary~2.8]{Banaji2022moran}, 
\[ \lid G_{p,d} \geq \frac{\theta \dim_\mathrm{A} G_{p,d}\  \underline{\dim}_\mathrm{B} G_{p,d} }{\dim_\mathrm{A} G_{p,d}  - (1-\theta) \underline{\dim}_\mathrm{B} G_{p,d} } \geq \frac{d\theta d/(p+1)}{d - (1-\theta) d/(p+1)} = \frac{d\theta}{p+\theta},\]
completing the proof. 
\end{proof}

Continuity of the intermediate dimensions at $\theta = 0$ has been shown in~\cite{Burrell2021,Burrell2022brownian} to have powerful consequences. We now use Proposition~\ref{lattice} to apply Burrell, Falconer and Fraser's result~\cite[Corollary~6.4]{Burrell2021} to give an application of Theorem~\ref{mainint} to orthogonal projections. Dimension theory of orthogonal projections has a long history in fractal geometry, see~\cite{Falconer2015:FractalsV,Shmerkin2015:FractalsV}. Recently there has been particular interest in orthogonal projections of dynamically defined sets, where one can often obtain more precise information than is provided by the general projection theorems, see~\cite{Hochman2012,Shmerkin2015:FractalsV}.  The following example falls into this category.  

\begin{example}\label{proj}
Let $p>0$ and consider a set of contracting similarity maps on $\mathbb{R}^2$ with fixed points lying in the set $G_{p,2}$ from Proposition~\ref{lattice}, with no two maps having the same fixed point. Assume the contraction ratios are small enough that the system forms a CIFS, with limit set $F$, say, and small enough that $\dim_\mathrm{H} F < 1$. Then by Theorem~\ref{mainint} and Corollary~\ref{lowerintcor}, 
\[ \dim_{\, \theta} F = \max\left\{\dim_\mathrm{H} F, \frac{2\theta}{p+\theta} \right\},\]
which is continuous at $\theta=0$. Therefore by~\cite[Corollary~6.4]{Burrell2021} and \cite[Theorem~1.8]{Falconer2020projections} there exists $c<1$ such that $\ubd \pi(F) \leq c$ for every orthogonal projection $\pi \colon \mathbb{R}^2 \to \mathbb{R}$, and $\ubd \pi(F) = c$ for almost every orthogonal projection $\pi$ (with respect to the natural measure). This conclusion is perhaps most interesting when $p$ is very close to 0 (and so $\bd F$ is very close to 2) and $\dim_\mathrm{H} F$ is very close to 1. 
\end{example}
More generally, if $1 \leq n < d$ are integers and the contraction ratios lie on $G_{p,d}$ and $\dim_\mathrm{H} F < n$, then $\ubd \pi(F) < n$ for every orthogonal projection $\pi \colon \Rd \to \mathbb{R}^n$.

\section{Continued fraction sets}\label{ctdfracsect}

\subsection{Real continued fractions}

In this section we apply Theorem~\ref{mainint} to give information about sets of irrational numbers whose continued fractions have restricted entries, as in the following definition. 

\begin{defn}\label{ctdfracdefn}
For a non-empty, proper subset $I \subset \N$, define 
\[ F_I \coloneqq \left\{ \, z \in (0,1) \setminus \mathbb{Q} : z = \frac{1}{b_1 + \frac{1}{b_2 + \frac{1}{\ddots}}}, b_n \in I \mbox{ for all } n \in \N \, \right\}. \]
\end{defn}

The following lemma shows why our general results can be applied in this setting. 

\begin{lma}
Working in $\mathbb{R}$, letting $X \coloneqq [0,1]$ and $V \coloneqq (-1/8,9/8)$, 
\begin{enumerate}
\item\label{ctdfracnot1} If $1 \notin I$ then $\{ \, S_b(x) \coloneqq 1/(b+x) : b \in I \, \}$ is a CIFS with limit set $F_I$. 
\item\label{ctdfrac1} If $1 \in I$ then $\{ \, S_b(x) \coloneqq 1/(b+x) : b \in I, b \neq 1 \, \} \cup \left\{ \, S_{1b}(x) \coloneqq \frac{1}{b+\frac{1}{1+x}} : b \in I \, \right\}$ is a CIFS with limit set $F_I$. 
\end{enumerate}
\end{lma}
\begin{proof}
\eqref{ctdfracnot1} is verified in~\cite[page~4997]{Mauldin1999}, and~\eqref{ctdfrac1} can be verified similarly, noting that when $1 \in I$ the different CIFS is needed to ensure that the system is uniformly contractive, because the derivative of $x \to 1/(1+x)$ at $x=0$ is $-1$. 
\end{proof}

It follows from~\cite[Theorem~3.15]{Mauldin1996} that $\dim_\mathrm{H} F_I = h$; the Hausdorff dimension of such limit sets has been studied in~\cite{Mauldin1999,Kessebohmer2006,Heinemann2002,Ingebretson2020,Chousionis2020} %
 and other works. It follows from~\cite[Theorem~2.11]{Mauldin1999} that $\ubd F_I = \max\{h,\ubd\{ \, 1/b : b \in I \, \}\}$. In Theorem~\ref{ctdfracintthm} we apply Theorem~\ref{mainint} to give information about the intermediate dimensions of $F_I$. 

\begin{thm}\label{ctdfracintthm}
Using the notation in Definition~\ref{ctdfracdefn}, for any non-empty proper $I \subset \N$, 
\begin{enumerate}
\item\label{ctdfracint} For all $\theta \in [0,1]$ we have 
\begin{gather*}
 \uid F_I = \max\{h,\uid\{ \, 1/b : b \in I \, \}\}; \\*
 \max\{h,\lid\{ \, 1/b : b \in I \, \}\} \leq \lid F_I \leq \max\{h,\uid\{ \, 1/b : b \in I \, \}\}.
 \end{gather*}
\item\label{ctdfraccts} The maps $\theta \mapsto \uid F_I$ and $\theta \mapsto \lid F_I$ are continuous at $\theta = 0$.
\end{enumerate}
\end{thm} 

\begin{proof}

\eqref{ctdfracint} \textbf{Case 1:} Assume $1 \notin A$. Then the result follows from Theorem~\ref{mainint}~\eqref{inteq} and Corollary~\ref{lowerintcor}~\eqref{lowerint} if we take $P=\{ \, S_b(0) : b \in I \, \} = \{ \, 1/b : b \in I \, \}$. 

\textbf{Case 2:} Assume $1 \in A$. Since the map $x \mapsto 1/(1+x)$ is bi-Lipschitz on $[0,1]$ and $\uid$ is stable under bi-Lipschitz maps, 
\[\uid \left\{ \, \frac{1}{1+\frac{1}{b}} \, \right\} = \uid \{ \, 1/b : b \in I \, \}.\]
 Since the removal of finitely many points from a set does not change its dimension, \[\uid \{ \, 1/b : b \in I, b \neq 1 \, \} = \uid \{ \, 1/b : b \in I \, \}.\] It is clear from the definition that $\uid$ is finitely stable, so the equality for $\uid F_I$ follows from Theorem~\ref{mainint}~\eqref{inteq} if we take 
\[ P \coloneqq \left\{ \, \frac{1}{1+\frac{1}{b}} \, \right\} \cup \{ \, 1/b : b \in I, b \neq 1 \, \}.\] 
The lower bound holds since 
\[ \max\{h,\lid\{ \, 1/b : b \in I \, \}\} = \max\{ \dim_\mathrm{H} F_I, \lid \{ \, 1/b : b \in I, b \neq 1 \, \} \} \leq \lid F_I \]
by~\cite[Theorem~3.15]{Mauldin1996}. 

\eqref{ctdfraccts} For all $\theta \in (0,1]$, 
\begin{align*}
 \uid F_I &= \max\{h,\uid\{ \, 1/b : b \in I \, \}\} 
 &&\text{by \eqref{ctdfracint}} \\
 &\leq \max\{ h,\uid\{ \, 1/b : b \in \N \, \}\} \\
  &= \max\{ h, \theta/(1+\theta) \} &&\text{by \cite[Proposition~3.1]{Falconer2020}} \\ 
 &\xrightarrow[\theta \to 0^+]{} \max\{ h, 0 \} = h = \dim_\mathrm{H} F_I = \overline{\dim}_{\,0} F_I &&\text{by \cite[Theorem~3.15]{Mauldin1996}},
 \end{align*}
 so $\theta \mapsto \uid F_I$ is continuous at $\theta = 0$. Since $\dim_\mathrm{H} F_I \leq \lid F_I \leq \uid F_I$ for all $\theta \in [0,1]$ it follows that $\theta \mapsto \lid F_I$ is also continuous at $\theta = 0$. 
\end{proof}

The intermediate dimensions of the limit set of a CIFS will not generally be continuous at $\theta = 0$. For example, if the fixed points are arranged to be at $\frac{1}{\log n}$ then by Falconer, Fraser and Kempton's~\cite[Example~1]{Falconer2020}, the intermediate dimensions of the fixed points will be 1 for all $\theta \in (0,1]$, but the contraction ratios can be made small enough (and tending to 0 rapidly enough) so that $h < 1$. In this case, the limit set and its closure will have the same Hausdorff dimension as they differ by a countable set (namely the images of the point 0 under maps corresponding to finite words) by Mauldin and Urbański's~\cite[Lemma~2.1]{Mauldin1996}. Therefore by Banaji's~\cite[Theorem~6.1]{Banaji2023gen} the $\Phi$-intermediate dimensions can be used to `recover the interpolation' between the Hausdorff and box dimensions of the limit set $F$ in the sense that for all $s \in [h,1]$ there exists an admissible function $\Phi_s$ with $\dim^{\Phi_s} F = s$.  Theorem~\ref{mainint}~\eqref{phieq} can help to find these functions. 
It is possible for the intermediate dimensions to be discontinuous at $\theta = 0$ even when the box dimension is less than 1. Indeed, consider a countable compact subset $P \subset \mathbb{R}$ with box and Assouad dimension equal and strictly between 0 and 1, so by~\cite[Proposition~2.4]{Falconer2020}, $\dim_{\, \theta} P = \dim_\mathrm{B} P$ for all $\theta \in (0,1)$. We may then choose a set of similarity maps whose fixed points form the set $P$ and whose contraction ratios are small enough that the system forms a CIFS with the Hausdorff dimension of the limit set being smaller than $\dim_\mathrm{B} P$. Then the intermediate dimensions of the limit set will be discontinuous at $\theta = 0$ by Theorem~\ref{mainint}. 

The following example is similar to~\cite[Theorem~6.2]{Mauldin1999}; we consider a nice family of subsets $I$ which result in the upper and lower intermediate dimensions coinciding. 

\begin{example}\label{ctdfracex}
Fix $p>1$ and define 
\[ I_{p,l} \coloneqq \{ \, \lfloor n^p \rfloor : n \geq l \} \quad \mbox{for} \quad l \in \N, l \geq 2.\]
 Then $I_{p,l}$ is bi-Lipschitz equivalent to a cofinite subset of $\{ \, i^{-p} : i \in \N \, \}$, so 
 \[ \dim_{\, \theta} \{ \, 1/b : b \in I_{p,l} \, \} = \dim_{\, \theta} \{ \, i^{-p} : i \in \N \, \} = \frac{\theta}{p+\theta}\]
 for $\theta \in [0,1]$ by~\cite[Proposition~3.1]{Falconer2020}. Therefore the intermediate dimensions of the continued fraction set are 
  \[ \dim_{\, \theta} F_{I_{p,l}} = \max\left\{\dim_{\mathrm{H}} F_{I_{p,l}},\frac{\theta}{p+\theta} \right\}\] by the bounds in Theorem~\ref{ctdfracintthm}~\eqref{ctdfracint}, which coincide in this case. Recall that it was shown in~\cite[Section~3]{Mauldin1996} that $\theta_S = \inf\{ \, t > 0 : \psi_1(t) < \infty \, \}$. But there exists $C \geq 1$ such that $1/(Cb^2) \leq ||S_b'|| \leq C/b^2$ for all $b \in \N$. Therefore, since  
  \[ \sum_{n=1}^\infty ((n^p)^{-2})^{1/(2p)} = \sum_{n=1}^\infty n^{-1} = \infty;\] 
  \[ \sum_{n=1}^\infty ((n^p)^{-2})^s = \sum_{n=1}^\infty n^{-2p/s} < \infty \quad \mbox{for all} \quad s>1/(2p), \]
   it follows that the finiteness parameter $\theta_{I_{p,l}} = 1/(2p)$ and so $\dim_{\mathrm{H}} F_{I_{p,l}} > 1/(2p)$. But in~\cite[Theorem~1.5]{Mauldin1999} Mauldin and Urbański showed that $\theta_{I_{p,l}}$ is the infimum of the Hausdorff dimension of cofinite subsystems, so $\dim_{\mathrm{H}} F_{I_{p,l}} \to 1/(2p)$ as $l \to \infty$. Since $1/(2p) < 1/(p+1) = \dim_{\mathrm{B}} \{ \, 1/b : b \in I_{p,l} \, \}$, there exists $q \in \N$ such that 
   \[ \dim_{\mathrm{H}} F_{I_{p,l}} < \dim_{\mathrm{B}} F_{I_{p,l}} = 1/(p+1) \quad \mbox{for all} \quad l \geq q. \] 
The graph of the intermediate dimensions in the case $p=2$ and $l$ large enough that $1/4 < \dim_{\mathrm{H}} F_{I_{p,l}} < \dim_{\mathrm{B}} F_{I_{p,l}} = 1/3$ is the black curve in Figure~\ref{fig:holder}. This example shows that Theorems~\ref{mainint} and \ref{ctdfracintthm} provide examples of sets for which the intermediate dimensions are constant and equal to the Hausdorff dimension until a phase transition, and then strictly increasing, concave and analytic; this form has not previously been seen in `natural' sets. The fact that the phase transition takes place at $\theta = \frac{2h}{1-h}$ means that for the dimension to increase above the Hausdorff dimension, the sizes of the covering sets need to be restricted to lie in intervals that are smaller than $[\delta,\delta^{\frac{2h}{1-h}}]$. Note that the graph of the intermediate dimensions is neither convex on the whole domain nor concave on the whole domain. 
\end{example}

The intermediate dimensions can also be used to give information about the H{\"o}lder distortion of maps between continued fraction sets (see~\cite[Section~17.10]{Fraser2020} for discussion of the H{\"o}lder mapping problem in the context of dimension theory). 
In~\cite[Section~14.2.1 5.]{Falconer2021-2} Falconer noted that if $F \subset \Rd$ is non-empty and bounded and $f \colon F \to \Rd$ is an $\alpha$-H{\"o}lder map for some $\alpha \in (0,1]$, meaning that there exists $c>0$ such that $||f(x)-f(y)|| \leq c||x-y||^\alpha$ for all $x,y \in F$, then 
\begin{equation}\label{generalholderint} \uid f(F) \leq \alpha^{-1} \uid F 
\end{equation}
 and $\lid f(F) \leq \alpha^{-1} \lid F$ (see~\cite[Theorem~3.1]{Burrell2022brownian} and~\cite[Section~4]{Banaji2023gen} for further H{\"o}lder distortion estimates). 
Example~\ref{holderint} is the first time that the intermediate dimensions for $\theta \in (0,1)$ have been observed to give better information about H{\"o}lder exponents than either the Hausdorff or box dimensions (for example in the case of elliptical polynomial spirals~\cite{Burrell2022spirals} the box dimension gives the best information). Recently, Banaji and Kolossv\'ary showed that the intermediate dimensions can also give better information than the Hausdorff or box dimensions for the H\"older distortion of certain Bedford--McMullen carpets, see~\cite[Example~2.12]{Banaji2021bedford}. 
Fraser~\cite{Fraser2019-4} showed that for the spiral winding problem, another spectrum of dimensions (the Assouad spectrum) gives better information about H{\"o}lder exponents than has been obtained from either of the two dimensions (the upper box and Assouad dimensions) that it interpolates between. 

\begin{example}\label{holderint}

Let $p,q$ be such that $1<p<q < 2p-1 < \infty$. As in Example~\ref{ctdfracex} there exists $l \in \N$ large enough so that if $I_{p,l} \coloneqq \{ \, \lfloor n^p \rfloor : n \geq l \}$ then 
\[ 1/(2p) < h_p < 1/(q+1) < 1/(p+1)\]
 where $h_p$ is the Hausdorff dimension of the continued fraction set $\dim_\mathrm{H} F_{I_{p,l}}$, and then $\dim_{\, \theta} F_{I_{p,l}} = \max \left\{ h_p, \frac{\theta}{p+\theta}\right\}$. Similarly, if $I_q$ is any subset of $\N$ whose symmetric difference with $\{ \, \lfloor n^q \rfloor : n \in \N \}$ is finite then $\dim_{\, \theta} F_{I_q} = \max \left\{ h_q, \frac{\theta}{q+\theta} \right\}$, where $h_q \coloneqq \dim_\mathrm{H} F_{I_q}$. If $I_q$ is also such that $h_q \in \left( \frac{ph_p}{q-qh_p+ph_p},\frac{1}{q+1} \right)$ and $f \colon F_{I_q} \to \mathbb{R}$ is an $\alpha$-H{\"o}lder map such that $f(F_{I_q}) \supseteq F_{I_{p,l}}$ then~\eqref{generalholderint} gives the best upper bound for $\alpha$ when $\theta = \frac{qh_q}{1-h_q}$, when 
\[ \alpha^{-1} h_q = \alpha^{-1} \dim_{\, \theta} F_{I_q} \geq  \uid f(F_{I_q}) \geq \dim_{\, \theta} F_{I_{p,l}} = \frac{\theta}{p+\theta} = \frac{qh_q}{p-ph_q + qh_q},\]
and so \[ \alpha \leq \frac{p-ph_q + qh_q}{q}.\] 
Using the Hausdorff dimension merely gives that $\alpha \leq h_q/h_p$, and the box dimension merely gives $\alpha \leq \frac{p+1}{q+1}$. The intermediate dimensions of the two sets and the upper bound with a certain choice of parameters are plotted in Figure~\ref{fig:holder}. 
 
\end{example}

\begin{figure}[ht]
\center{\includegraphics[width=0.75\textwidth]
        {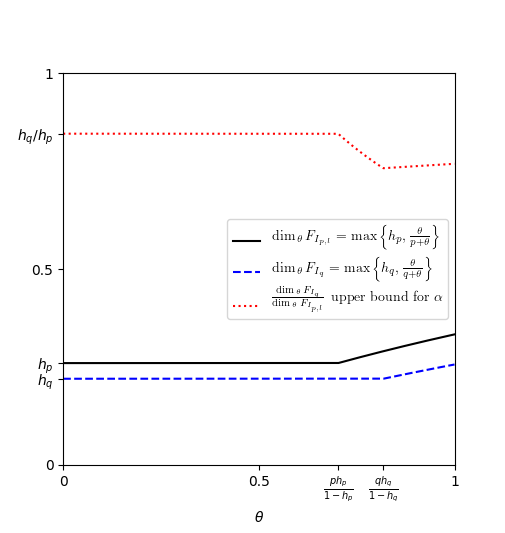}}
        \caption{\label{fig:holder}
        Graph of the intermediate dimensions of the real continued fraction sets in Example~\ref{holderint} and the upper bound for $\alpha$ against $\theta$ in the case $p=2$, $q=2.9$, $h_p \approx 0.26$, $h_q \approx 0.22$. 
 }
\end{figure}

Fractional Brownian motion is an important stochastic process, defined and studied in Kahane's classical text~\cite{Kahane1985}. 
Falconer~\cite{Falconer2021} has computed explicitly the intermediate dimensions of fractional Brownian images of certain sequence sets. 
In the following corollary of Theorem~\ref{ctdfracintthm} we apply results of Burrell~\cite{Burrell2022brownian} to give some consequences of the continuity of the intermediate dimensions of continued fraction sets for dimensions of images of $F_I$ under index-$\alpha$ fractional Brownian motion. Recall that $h = \dim_{\mathrm H} F_I$. Perhaps the most interesting part of Corollary~\ref{brownian} is the sufficiency of the condition $\alpha > h$ for the upper box dimension of the image to be strictly less than 1; this is an example of how the intermediate dimensions can be used to obtain information about the box dimension of sets. 

\begin{cor}\label{brownian}
Let $\alpha \in (0,1)$ and let $B_\alpha \colon \mathbb{R} \to \mathbb{R}$ denote index-$\alpha$ fractional Brownian motion. Then for any non-empty, proper subset $I \subset \N$, 
\begin{enumerate}
\item\label{browncts} The maps $\theta \mapsto \uid B_\alpha(F_I)$ and $\theta \mapsto \lid B_\alpha(F_I)$ are almost surely continuous at $\theta = 0$. %
\item\label{brownlarge} If $\alpha > h$ then almost surely 
\[ h/\alpha = \dim_\mathrm{H} B_\alpha(F_I) \leq \ubd B_\alpha(F_I) < 1.\]  
\item\label{brownsmall} If $\alpha \leq h$ then almost surely 
\[ \dim_\mathrm{H} B_\alpha(F_I) = \dim_{\mathrm{B}} B_\alpha(F_I) = 1.\] 
\end{enumerate}
\end{cor}

\begin{proof}

\eqref{browncts} follows immediately from Theorem~\ref{ctdfracintthm}~\eqref{ctdfraccts} and Burrell's result~\cite[Corollary~3.5]{Burrell2022brownian}. 

\eqref{brownlarge} The set $F_I$ is the limit set of a CIFS so it is Borel (see the discussion after Lemma~\ref{fitin}), so Kahane's general results~\cite[Chapter~18]{Kahane1985} %
give $h/\alpha = \dim_\mathrm{H} B_\alpha(F_I)$ almost surely. The middle inequality is a general property of the dimensions, and $\ubd B_\alpha(F_I) < 1$ almost surely by Theorem~\ref{ctdfracintthm}~\eqref{ctdfraccts} and Burrell's result~\cite[Corollary~3.7]{Burrell2022brownian}. 

\eqref{brownsmall} follows from Kahane's~\cite[Chapter~18]{Kahane1985} since $F_I$ is Borel. 
\end{proof}

Note that since $\dim_\mathrm{P} B_\alpha(F_I), \lbd B_\alpha(F_I) \in (\dim_\mathrm{H} B_\alpha(F_I),\ubd B_\alpha(F_I))$, if $\alpha > h$ then almost surely $\dim_\mathrm{P} B_\alpha(F_I) < 1$ and $\lbd B_\alpha(F_I) < 1$. On the other hand, if $\alpha \leq h$ then almost surely $\dim_\mathrm{P} B_\alpha(F_I) = \lbd B_\alpha(F_I) = 1$.

\subsection{Complex continued fractions}\label{compsect}

In this section we study sets of complex numbers which have a complex continued fraction expansion with restricted entries. For a non-empty $I \subseteq \{ \, m + n i : m \in \N, n \in \mathbb{Z} \, \}$, define 
\[ F_I \coloneqq \left\{ \, z \in \mathbb{C} : z = \frac{1}{b_1 + \frac{1}{b_2 + \frac{1}{\ddots}}}, b_n \in I \mbox{ for all } n \in \N \, \right\}. \]
If $1 \notin I$ then it can be verified, as in~\cite[Section~6]{Mauldin1996}, that if $1 \notin I$ then $\{ \, S_b(z) \coloneqq 1/(b+z) : b \in I \, \}$ is a CIFS with limit set $F_I$, with $X \subset \mathbb{C}$ being the closed disc centred at 1/2 with radius 1/2, and $V = B(1/2,3/4)$. If $1 \in I$ then $S_1$ is not uniformly contracting but it is straightforward to verify that $\{ \, S_b(z) \coloneqq 1/(b+z) : b \in I, b \neq 1 \, \} \cup \left\{ \, S_{1b}(z) \coloneqq \frac{1}{b+\frac{1}{1+z}} : b \in I \, \right\}$ is a CIFS with the same limit set. By~\cite[Theorem~3.15]{Mauldin1996}, the Hausdorff dimension can be determined by the topological pressure function, and has been studied in~\cite{Hanus1998} with estimates given in~\cite[Section~6]{Mauldin1996} and~\cite{Gardner1983,Priyadarshi2016,Falk2018,Ingebretson2020}. 
\begin{thm}\label{compmain}
Using the notation above, for all $\theta \in [0,1]$, 
\begin{enumerate}
\item\label{compctdfracint} For all $\theta \in [0,1]$ we have 
\begin{gather*}
 \uid F_I = \max\{h,\uid\{ \, 1/b : b \in I \, \}\}; \\
 \max\{h,\lid\{ \, 1/b : b \in I \, \}\} \leq \lid F_I \leq \max\{h,\uid\{ \, 1/b : b \in I \, \}\}.
 \end{gather*}
\item\label{compctdfraccts} The maps $\theta \mapsto \uid F_I$ and $\theta \mapsto \lid F_I$ are continuous at $\theta = 0$.
\end{enumerate}
\end{thm}

\begin{proof}

We sketch the proof as it is similar to the proof of Theorem~\ref{ctdfracintthm}. 

\eqref{compctdfracint} follows from Theorem~\ref{mainint}. 

\eqref{compctdfraccts} follows from~\eqref{compctdfracint} since $\{ \, 1/b : b \in I \, \}$ is the disjoint union of bi-Lipschitz copies of two subsets of $G_{1,2}$, whose intermediate dimensions are continuous by Proposition~\ref{lattice}. 
\end{proof}

\begin{example}\label{complexexample}
For $p \in (1,\infty)$ and $R \in [0,\infty)$ let $I_{p,R} \coloneqq \{ \, \lfloor m^p \rfloor + \lfloor n^p \rfloor i : n,m \in \N \, \} \setminus B(0,R)$. Then the set $\{ \, 1/b : b \in I \, \}$ is bi-Lipschitz equivalent to a cofinite subset of the set $G_{p,2}$ from Proposition~\ref{lattice}, so 
\[ \dim_{\, \theta} \{ \, 1/b : b \in I \, \} = \dim_{\, \theta} G_{p,2} = \frac{2\theta}{p+\theta}.\] 
Therefore the bounds in Theorem~\ref{compmain}~\eqref{compctdfracint} coincide and 
\[ \dim_{\, \theta} F_{I_{p,R}} = \max\left\{ \dim_\mathrm{H} F_{I_{p,R}} , \frac{2\theta}{p+\theta} \right\}. \] 
For all $t \geq 0$, writing $\simeq$ to mean up to multiplication by a positive, finite function of $t$, $p$ and $R$ and/or addition by a real-valued function of $t$, $p$ and $R$, and using the convention that $a \infty = \infty + c = \infty$ for $a \in (0,\infty)$ and $c \in \mathbb{R}$, we have
 \begin{align*}
  \psi_1(t) = \sum_{b \in I_{p,R}} ||S_b||^t &\simeq \sum_{b \in I_{p,R}} |b|^{-2t} \qquad \qquad \text{(Koebe distortion theorem)} \\
  &= \sum_{n = 0}^\infty \sum_{\substack{b \in I_{p,R} \\ 2^n \leq |b| < 2^{n+1}}} |b|^{-2t} \\
  &\simeq \sum_{n = 0}^\infty \# \{ \, b \in I_{p,R} : 2^n \leq |b| < 2^{n+1} \, \} (2^n)^{-2t} \\
  &\simeq \sum_{n=0}^\infty (2^{n/p})^2 (2^n)^{-2t} \\
  &= \sum_{n=0}^\infty 4^{n(p^{-1} - t)}. 
  \end{align*}
  Therefore the finiteness parameter $\theta_{I_{p,R}} = 1/p$. By~\cite[Theorem~1.5]{Mauldin1999}, $\dim_\mathrm{H} F_{I_{p,R}} \to 1/p$ as $R \to \infty$, so there exists $R' > 0$ such that for all $R > R'$, $\dim_\mathrm{H} F_{I_{p,R}} < \dim_\mathrm{B} F_{I_{p,R}} = 2/(p+1)$. 
\end{example}

Again we have consequences for fractional Brownian images. 

\begin{cor}
Let $\alpha \in (0,1)$ and let $B_\alpha \colon \mathbb{C} \to \mathbb{C}$ denote index-$\alpha$ fractional Brownian motion (identifying $\mathbb{C}$ with $\mathbb{R}^2$). Then for any non-empty $I \subseteq \{ \, m + n i : m \in \N, n \in \mathbb{Z} \, \}$, 
\begin{enumerate}
\item The maps $\theta \mapsto \uid B_\alpha(F_I)$ and $\theta \mapsto \lid B_\alpha(F_I)$ are almost surely continuous at $\theta = 0$. %
\item If $\alpha > (\dim_\mathrm{H} F_I)/2$ then almost surely \[(\dim_\mathrm{H} F_I)/\alpha = \dim_\mathrm{H} B_\alpha(F_I) \leq \ubd B_\alpha(F_I) < 2.\]
\item If $\alpha \leq (\dim_\mathrm{H} F_I)/2$ then almost surely $\dim_\mathrm{H} B_\alpha(F_I) = \dim_{\mathrm{B}} B_\alpha(F_I) = 2$.
\end{enumerate}
\end{cor}

\begin{proof}
Follows from Theorem~\ref{compmain} in a similar way to how Corollary~\ref{brownian} follows from Theorem~\ref{ctdfracintthm}. 
\end{proof}

Since $\dim_\mathrm{P} B_\alpha(F_I), \lbd B_\alpha(F_I) \in (\dim_\mathrm{H} B_\alpha(F_I),\ubd B_\alpha(F_I))$, if $\alpha > (\dim_\mathrm{H} F_I)/2$ then almost surely $\dim_\mathrm{P} B_\alpha(F_I) < 2$ and $\lbd B_\alpha(F_I) < 2$, whereas if $\alpha \leq (\dim_\mathrm{H} F_I)/2$ then almost surely $\dim_\mathrm{P} B_\alpha(F_I) = \lbd B_\alpha(F_I) = 2$. 
Note that it follows from Proposition~\ref{lattice} and~\cite[Corollary~3.7]{Burrell2022brownian} that for all $d \in \N$, $p \in (0,\infty)$ and $\alpha \in (0,1)$, if $B_\alpha : \Rd \to \Rd$ is index-$\alpha$ fractional Brownian motion on $\Rd$ then $\ubd B_\alpha(G_{p,d}) < d$.

\section{Generic attractors}\label{genericsect}

\subsection{Background and motivation}

Often, and especially in non-conformal settings and in the presence of overlaps, it is difficult to compute the dimension of a particular IFS attractor.  In a seminal paper from 1988 ~\cite{Falconer1988} Falconer introduced the idea of studying the generic dimension of a (finitely generated) self-affine set by fixing a set of matrices and then randomising the translations in a  suitable way. It turns out that, for Lebesgue almost every choice of translations, the Hausdorff and box dimension of the associated self-affine set are given by the affinity dimension: a dimension formula depending only on the matrices. 

K\"aenm\"aki and  Reeve~\cite{Kaenmaki2014} extended Falconer's  approach to the theory of infinitely generated self-affine sets.  Here one needs to randomise infinitely many translations, and do so in a manner which outputs a bounded set.  As such,  the natural space to draw from is $V^\mathbb{N}$ for some bounded set $V \subseteq \Rd$ with positive $d$-dimensional Lebesgue measure $\mathcal{L}^d(V)>0$.  For convenience from now on we assume $V= [0,1)^d$.  The space $V^{\N}$ carries a natural infinite product probability measure
\[
\mu = \prod_{i \in \mathbb{N}} \mathcal{L}^d \vert_V.
\]
 K\"aenm\"aki and  Reeve proved that if one fixes an infinite collection of strictly contracting matrices on $\Rd$ and randomises the translations according to $\mu$, then  the Hausdorff dimension of the associated attractor is almost surely given by the natural extension of the affinity dimension to the infinite case.  In comparison with Falconer's result, this notably omits the box dimension.  There is good reason for this since the box dimension of an infinitely generated attractor depends much more sensitively on the translations themselves, as we have seen above. 

We show here that almost surely the box dimension is $d$, that is, the ambient spatial dimension.  In fact we show more.  We show that for an arbitrary IIFS (not necessarily consisting of affine maps) the associated attractor is generically somewhere dense, and therefore the box and intermediate dimensions are all generically equal to $d$. 
 Moreover, \emph{generically} can refer to either $\mu$ almost surely, or for a comeagre set of translates (topologically generic).

\subsection{Results}

Fix an IIFS $\{S_i\}_{i \in \mathbb{N}}$ defined on $[0,2]^d$ with the property that $S_i([0,2]^d) \subseteq [0,1]^d$ for all $i \in \mathbb{N}$. For $t = (t_1, t_2, \dots) \in V^\mathbb{N}$ let  $F_t$ denote the attractor of the IIFS $\{S_i+t_i\}_{i \in \mathbb{N}}$ with $S_i+t_i$ defined on $[0,2]^d$ by $(S_i+t_i)(x) = S_i(x)+t_i$. We assume throughout that the contraction ratios of the maps $S_i$ only accumulate at zero. The assumption that $[0,2]^d$ maps into $[0,1]^d$ is to ensure the maps can be composed with translations in a well-defined way. 

Write $\fix(g)$ to denote the unique fixed point of a contraction $g$ on $[0,2]^d$.  We use the following simple lemma to relate random translations to random fixed points. Fixed points are useful because they necessarily belong to the attractor. In what follows balls are open. 

\begin{lma}\label{genericlma}
Let $u \in V$, $g$ be a contraction on $[0,2]^d$ with $g([0,2]^d) \subseteq [0,1]^d$, $q \in [0,2]^d$, and $\delta>0$. Then
\[
\fix(g+u) \in B(q,\delta) \Leftrightarrow u \in B(q-g(\fix(g+u)), \delta).
\]
\end{lma}

\begin{proof}
By definition of a fixed point
\begin{align*}
\fix(g+u) \in B(q,\delta) &\Leftrightarrow ||\fix(g+u) - q || < \delta \\
&\Leftrightarrow ||g(\fix(g+u))+u - q || < \delta \\
&\Leftrightarrow u \in B(q-g(\fix(g+u)), \delta)
\end{align*}
as required.
\end{proof}

\begin{thm}\label{measure}
For $\mu$-almost all $t \in V^\mathbb{N}$, the attractor $F_t$ is somewhere dense in $[0,2]^d$, so for all $\theta \in (0,1]$,
\[
\dim_{\, \theta} F_t = \dim_\mathrm{B} F_t = d.
\]
\end{thm}

\begin{proof}
Let $z \in [0,1]^d$ be an accumulation point of the set $\{\fix(S_i)\}_i$. We prove that $F_t$ is almost surely dense in the explicit (unit) square $V + z$. The idea is that for infinitely many $i \in \N$ the fixed point of $S_i$ will be close to $z$ and the contraction ratio of $S_i$ will be small, so if the fixed point of $S_i + t_i$ is far away from a given point $q \in V + z$ then $t_i$ must be far away from $q-z$, which we use to bound the measure. We have 
\begin{align*}
&\left\{ \, t \in V^\mathbb{N} :  F_t \text{ is nowhere dense} \, \right\}  \subseteq    \left\{ \, t \in V^\mathbb{N} :  F_t \text{ is not dense in $V+z$} \, \right\} \\
& \subseteq \bigcup_{q \in (V+z) \cap \mathbb{Q}^d} \bigcup_{\delta  \in \mathbb{Q}^+} \left\{ \, t \in V^\mathbb{N} :  \forall i \in \mathbb{N} , \, \fix(S_i+t_i) \notin  B(q,\delta)  \, \right\} \\ 
& =  \bigcup_{q \in (V+z) \cap \mathbb{Q}^d} \bigcup_{\delta  \in \mathbb{Q}^+}   \left\{ \, t \in V^\mathbb{N} :  \forall i \in \mathbb{N} , \,  t_i  \notin  B(q-S_i(\fix(S_i +t_i)),\delta)  \, \right\} &\text{(Lemma~\ref{genericlma})}\\ 
&= \bigcup_{q \in (V+z) \cap \mathbb{Q}^d} \bigcup_{\delta  \in \mathbb{Q}^+} \bigcap_{N=1}^\infty T_{q,\delta, N}
\end{align*}
where for $N \in \mathbb{N}$,
\[
T_{q,\delta, N} \coloneqq \left(\prod_{i=1}^N \{ \, t \in V : || q-S_i(\fix(S_i +t)) - t || \geq \delta \, \} \right) \times \left(\prod_{j=N+1}^\infty V \right).
\]
For each $i \in \N$, $\{ \, t \in V : || q-S_i(\fix(S_i +t)) - t || \geq \delta \, \}$ is a Borel set as it is the preimage of $[\delta,\infty)$ by a continuous function, so each $T_{q,\delta, N}$ is $\mu$-measurable. 
By definition of $z$ we can find infinitely many $i \in \mathbb{N}$ such that the maximum of $||\fix(S_i) - z||$ and the contraction ratio of $S_i$ is less than $\delta/10$. 
For any such $i \in \N$, if $t \in B(q-z,\delta/2) \cap V$ then 
\begin{align*} 
|| q-S_i(\fix(S_i + t)) - t || &\leq ||q-z-t|| + ||z-\fix(S_i)|| + ||\fix(S_i) - S_i(\fix(S_i + t))|| \\
&< \frac{\delta}{2} + \frac{\delta}{10} + 2\frac{\delta}{10} < \delta.
\end{align*}
Therefore for infinitely many $i \in \mathbb{N}$ we have  
\[
\mathcal{L}^d \Big( \{ t \in V : || q-S_i(\fix(S_i +t)) - t || \geq \delta \} \Big) \leq 1-2^{-d} \left(\frac{\min\{\delta,1/2\}}{2}\right)^d k_d < 1, 
\] %
where $k_d$ is the $d$-dimensional Lebesgue measure of a ball in $\Rd$ of unit radius.
This uniform bound away from 1 is independent of $i$, so for all $q$ and $\delta$, $\mu(T_{q,\delta, N}) \to 0$ as $N \to \infty$, so $\bigcap_{i=1}^\infty T_{q,\delta, N}$ is $\mu$-measurable (as the countable intersection of $\mu$-measurable sets) with $\mu\left(\bigcap_{i=1}^\infty T_{q,\delta, N}\right) = 0$. 
Therefore 
\[ \bigcup_{q \in (V+z) \cap \mathbb{Q}^d} \bigcup_{\delta  \in \mathbb{Q}^+} \bigcap_{i=1}^\infty T_{q,\delta, N}\]
 is a countable union of $\mu$-measurable sets with $\mu$-measure 0, so it is itself $\mu$-measurable with $\mu$-measure 0, which proves the result. 
\end{proof}

A simple modification of the above argument proves that $F_t$ is somewhere dense for a prevalent set of $t \in V^\mathbb{N}$. 
In finite dimensional Banach spaces, being prevalent is equivalent to having full Lebesgue measure, but the notion of prevalence applies to more general topological groups as well, see~\cite{Ott2005}. 
To equip $V^\mathbb{N}$ with a topological group structure, we take as the infinite product of the group $V$ under addition mod 1 with the product topology.  A subset $A$ of a topological group $G$ is \emph{prevalent} if the complement of $A$ is contained in a Haar-null set.  A \emph{Haar-null} set is a Borel set $B \subseteq G$ for which there exists a Borel probability measure $\nu$ on $G$ such that  $\nu(gBh) = 0$ for all $g,h \in G$ (for abelian groups, $gBh$ can be replaced by $gB$, see~\cite{Elekes2020}). 

\begin{thm}\label{prevalent}
For a prevalent set of $t \in V^\mathbb{N}$, the attractor $F_t$ is somewhere dense in $[0,2]^d$.  In particular, for all $\theta \in (0,1]$,
\[
\dim_{\, \theta} F_t = \dim_\mathrm{B} F_t = d.
\]
\end{thm}

\begin{proof}
Let $B= \{ t \in V^\mathbb{N} : F_t \text{ is nowhere dense}$\}.  This is easily seen to be a Borel set. Choose $\nu = \mu$, where $\mu$ is as above.  The proof that $\mu(gB) = 0$ for all $g \in  V^\mathbb{N}$ then proceeds as in the proof of Theorem~\ref{measure} using the fact that $\mathcal{L}^d \vert_V$ is invariant under the action of $V$.
\end{proof}

Finally, we establish a topological result. We endow $V^\mathbb{N}$ with the Hilbert cube metric
\[
d(t, s) = \left(\sum_{i=1}^\infty \frac{||t_i-s_i||^2}{i^2}\right)^{1/2},
\]
noting that this generates the product topology on $V^\mathbb{N}$.
Recall that a subset of $V^{\N}$ is called \emph{residual} or \emph{comeagre} (topologically generic) if it contains a countable intersection of open dense sets. 
Note that $V$ is homeomorphic to the separable complete metric space $[0,\infty)$, so $V$ is a Polish space, hence the countable product $V^{\N}$ is also a Polish space. The Baire Category Theorem therefore implies that $V^{\N}$ is a Baire space, meaning that residual subsets are dense. 

\begin{thm}\label{baire}
For a residual  set of $t \in V^\mathbb{N}$, the attractor $F_t$ is somewhere dense in $[0,2]^d$.  In particular, for all $\theta \in (0,1]$,
\[
\dim_{\, \theta} F_t = \dim_\mathrm{B} F_t = d.
\]
\end{thm}

\begin{proof}
Let $z \in [0,1]^d$ be as in the proof of Theorem~\ref{measure}. Then
\begin{align*}
& \left\{ \, t \in V^\mathbb{N} :  F_t \text{ is somewhere dense} \, \right\}   \\ 
& \supseteq   \left\{ \, t \in V^\mathbb{N} :  F_t \text{ is  dense in $V+z$} \, \right\} \\  
& =  \bigcap_{q \in (V+z) \cap \mathbb{Q}^d} \bigcap_{\delta  \in \mathbb{Q}^+}  \left\{ \, t \in V^\mathbb{N} :  \exists i \in \mathbb{N}, \, \fix(S_i+t_i) \in  B(q,\delta) \, \right\} \\ 
& =  \bigcap_{q \in (V+z) \cap \mathbb{Q}^d} \bigcap_{\delta  \in \mathbb{Q}^+}  \left\{ \, t \in V^\mathbb{N} :  \exists i \in \mathbb{N}, \, t_i  \in  B(q-S_i(\fix(S_i+t_i)) ,\delta)  \, \right\} &\text{(Lemma~\ref{genericlma})}\\ 
& \eqqcolon  \bigcap_{q \in (V+z) \cap \mathbb{Q}^d} \bigcap_{\delta  \in \mathbb{Q}^+} T_{q,\delta}.
\end{align*}
Fix $q \in (V+z) \cap \mathbb{Q}^d $ and $ \delta  \in \mathbb{Q}^+$. The set $T_{q,\delta}$ is immediately seen to be open since $\fix(S_i+t_i)$ is continuous in $t_i$ and we use open balls. Moreover, it is also dense since an element $t \in V^{\mathbb{N}}$ may be approximated arbitrarily well in the metric $d$ within the set $T_{q,\delta}$ by replacing $t_i$ with $q-z$ for sufficiently large $i$.
\end{proof}

\begin{rem}\label{modifycontractions}
In the above setting, all the contraction ratios were bounded above by $1/2$, but we can easily avoid this.
Indeed, fix any $c>0$ and fix any IIFS $\{S_i\}_{i \in \N}$ of contractions defined on $[0,1+c]^d$ satisfying $S_i([0,1+c]^d) \subseteq [0,1]^d$ for all $i \in \N$. Then the contraction ratios are bounded above by $\frac{1}{1-c}$ and we assume they accumulate only at 0. We can let $V \coloneqq [0,c)^d$, and again $V^\N$ can be equipped with a natural probability measure by taking the infinite product of the Lebesgue measure on $V$ and then normalising. Moreover, $V^\N$ can be equipped with a topological group structure by taking the infinite product of the group $V$ under addition mod $c$ with the product topology. Then if $t = (t_1,t_2,\ldots) \in V^\N$ then $\{S_i + t_i\}_{i \in \N}$ is an IIFS of contractions on $[0,1+c]^d$, with limit set $F_t$, say. Similar proofs show that under the assumptions of any one of Theorem~\ref{measure}, \ref{prevalent} or \ref{baire}, $F_t$ is somewhere dense in $[0,1+c]^d$, so the same conclusions about dimensions hold. 
\end{rem}

Under the assumptions of any one of Theorem~\ref{measure}, \ref{prevalent} or \ref{baire}, either in the setting of those theorems or in the more general setting of Remark~\ref{modifycontractions}, the attractor $F_t$ is somewhere dense. This means that if $\dim$ is any notion of dimension which is stable under closure, for example any of the $\Phi$-intermediate or Assouad-type dimensions, then $\dim F_t = d$. 
If all of the $S_i$ are bi-Lipschitz, then by a result of Mauldin and Urbański~\cite[Theorem~3.1]{Mauldin1996} the upper box and packing dimensions of the attractor coincide, so under the assumptions of any one of Theorem~\ref{measure}, \ref{prevalent} or \ref{baire} we have $\dim_\mathrm{P} F_t = d$.

\section*{Acknowledgements}
We thank Lars Olsen, P\'eter Varj\'u, and an anonymous referee for some helpful comments. 
Both authors were financially supported by a Leverhulme Trust Research Project Grant (RPG-2019-034), and JMF was also supported by an EPSRC Standard Grant (EP/R015104/1).

\section*{References}

\printbibliography[heading=none]

\end{document}